\documentclass[a4paper,reqno,10pt,oneside]{amsart}

\usepackage{geometry}
\usepackage{amsmath}
\usepackage{paralist}
\usepackage{graphics} 
\usepackage{epsfig}
\usepackage{graphicx}
\usepackage{epstopdf}
\usepackage[colorlinks=true]{hyperref}

\hypersetup{urlcolor=blue, citecolor=red}
\newtheorem{theorem}{Theorem}[section]

\newtheorem{lemma}[theorem]{Lemma}
\newtheorem{proposition}[theorem]{Proposition}

\theoremstyle{definition}

\newtheorem{cla}{Step}
\newcommand{\ep}{\varepsilon}

\newcommand{\complement}{\mathsf{c}}
\newcommand{\R}{\mathbb{R}}
\newcommand{\N}{\mathbb{N}}
\newcommand{\D}{\mathcal{D}}
\newcommand{\de}{\,\mathrm{d}}
\numberwithin{equation}{section}

\title[Sign-changing bubble-tower solutions]{Sign-changing bubble-tower solutions to fractional semilinear elliptic problems}

\author[Gabriele Cora and Alessandro Iacopetti]{}


\email{gabriele.cora@unito.it}
\email{iacopetti@mat.uniroma1.it}

\thanks{\emph{Acknowledgements.} Research partially supported by the project ERC Advanced Grant 2013 n.~339958 Complex Patterns for Strongly Interacting Dynamical Systems COMPAT, and by Gruppo Nazionale per l'Analisi Matematica, la Pro\-ba\-bi\-li\-t\`a e le loro Applicazioni (GNAMPA) of the Istituto Nazionale di Alta Matematica (INdAM)}

\date{}

\begin{document}

\maketitle

\centerline{\scshape Gabriele Cora}
\medskip
{\footnotesize
 \centerline{Dipartimento di Matematica ``G. Peano"}
   \centerline{Universit\`a di Torino }
   \centerline{ via Carlo Alberto 10 -- 10123 Torino, Italy}
}

\medskip

\centerline{\scshape Alessandro Iacopetti}
\medskip
{\footnotesize
 \centerline{ Dipartimento di Matematica}
   \centerline{Universit\`a di Roma ``La Sapienza"}
   \centerline{P.le Aldo Moro 5 -- 00185 Roma, Italy}

}

\begin{abstract}
We study the asymptotic and qualitative properties of least energy radial sign-changing solutions to fractional semilinear elliptic problems of the form
\[
\begin{cases}
(-\Delta)^s u = |u|^{2^*_s-2-\varepsilon}u &\text{in } B_R, \\
u = 0 &\text{in }\R^n \setminus B_R,
\end{cases}
\]
where $s \in (0,1)$, $(-\Delta)^s$ is the s-Laplacian, $B_R$ is a ball of $\R^n$, $2^*_s := \frac{2n}{n-2s}$ is the critical Sobolev exponent and $\varepsilon>0$ is a small parameter.
We prove that such solutions have the limit profile of a ``tower of bubbles'', as $ \varepsilon \to 0^+$, i.e. the positive and negative parts concentrate at the same point with different concentration speeds. Moreover, we provide information about the nodal set of these solutions.
\end{abstract}

\begin{section}{Introduction}
Let $s \in (0,1)$, let $n \in \N$ be such that $n>2s$ and let $\Omega \subset \R^n$ be a bounded smooth domain. Consider the following non-local elliptic problem
\begin{equation}\label{fracsemilinear}
\begin{cases}
(-\Delta)^s u = f(u) &\text{in }\Omega, \\
u = 0 &\text{in }\R^n \setminus \Omega,
\end{cases}
\end{equation}
where $(-\Delta)^s$ is the s-Laplacian, $f(u)=|u|^{2^*_s-2-\varepsilon}u$ or $f(u)=\varepsilon u + |u|^{2^*_s-2}u$ for $n>6s$,  $\varepsilon>0$ is a small parameter and
$2^*_s := \frac{2n}{n-2s}$ is the critical exponent for the fractional Sobolev embedding.

\medskip

In the recent paper \cite{ChoiKimLee} the authors studied the asymptotic properties of least energy positive solutions to Problem \eqref{fracsemilinear}, i.e. positive solutions $u_\ep$ such that $\|u_\ep\|^2_s \to S_s^{\frac{n}{2s}}$, as $\varepsilon \to 0^+$, where $\|\cdot\|_s$ is the standard seminorm in $H^s(\R^n)$ and $S_s$ is the best fractional Sobolev constant. They proved, in the case of the spectral fractional Laplacian, that such solutions concentrate and blow-up at some point $x_0 \in \Omega$, providing also information about the blow-up speed with respect to $\varepsilon$. Their result is hence the fractional counterpart of the classical results of Han and Rey (see \cite{Han, Rey}) for the Laplacian.

Motivated by that, it is natural to ask whether is possible or not to extend to the fractional framework analogous results about the asymptotic behavior of least energy sign-changing solutions to almost critical and critical semilinear elliptic problems for the Laplacian (see \cite{BEP2,BEP3,IacPac1,IacVair1, IacVair2,PistWeth}). 

At first glance the answer seems to be positive, but differently from the case of constant-sign solutions, several difficulties arise when studying the qualitative properties of sign-changing solutions. Indeed, in view of the non-local interactions between the nodal components, we cannot take benefit from the fractional moving plane method (see \cite{ChenLiLi}), and the strong maximum principle does not work properly (see \cite[Sect. 1]{CorIac}). Moreover, when considering least energy sign-changing solutions, i.e. sign-changing solutions $u_\ep$ to \eqref{fracsemilinear} such that $\|u_\ep\|_s^2\to 2S_s^{\frac{n}{2s}}$, as $\varepsilon \to 0^+$, we cannot establish by mere energetic arguments, neither by a Morse-index approach, the number of  nodal components. In the local case it is well known that they possess exactly two nodal regions, since each nodal component carries the energy $S_1^{\frac{n}{2}}$ (see \cite{BEP2,BEP3}).
 In the fractional case we can only say that both the positive and the negative part globally carry the same energy $S_s^{\frac{n}{2s}}$, when $\varepsilon \to 0^+$, but this does not hold true in general for each individual nodal component and causes many troubles when performing the asymptotic analysis.
\medskip
   
In our contribution \cite{CorIac} we tackled the case of least energy radial sign-changing solutions to Problem \eqref{fracsemilinear} in a ball, when $f(u)=\ep u + |u|^{2_s^*-2}u$ is the critical nonlinearity and $n>6s$. In the spirit of the pioneering papers \cite{FrLe1,FrLe2}, we showed that these solutions change sign at most twice and exactly once when $s$ is close to $1$. Moreover, when $s>\frac{1}{2}$, we proved that they behave like a tower of two bubbles as $\ep\to 0^+$, namely, the positive and the negative part blow-up and concentrate at the same point (which is the center of the ball) with different speeds. Nevertheless, we needed to assume that these solutions change sign exactly once to determine which one between the positive and the negative part blew-up faster (see  \cite[Sect.1]{CorIac}). 

We point out that for $2s<n\leq 6s$, according to a classical result of Atkinson, Brezis, and Peletier (see \cite{ABP}), radial sign-changing solutions in a ball may not exist when $\ep>0$ is close to zero, while they do exist for $n>6s$ (see \cite[Theorem 3.7]{CorIac}).
\medskip

In this paper we consider slightly subcritical nonlinearities $f(u)=|u|^{2_s^*-2-\ep}u$, and we extend the results of \cite{CorIac} to all $s \in (0,1)$ without any extra assumption. The same proofs work also in the case of critical nonlinearities with minor modifications. The main result of our paper is the following:
\begin{theorem}\label{mainteo}
Let $s \in (0,1)$ and let $n>2s$. Let $(u_\ep)_\ep$ be a family of least energy radial sign-changing solutions to \begin{equation}\label{fracradsubcrit}
\begin{cases}
(-\Delta)^s u = |u|^{2^*_s-2-\varepsilon}u &\text{in } B_R, \\
u = 0 &\text{in }\R^n \setminus B_R,
\end{cases}
\end{equation}
where $B_R$ is the euclidean ball of radius $R>0$ centered at the origin. Assume without loss of generality that $u_\ep(0)> 0$ and set $M_\ep^\pm:=|u_\ep^\pm|_\infty$. 
Then, as $\ep \to 0^+$ it holds that:
\begin{itemize}
\item[(i)] $M_\ep^\pm \to + \infty$,
\item[(ii)] $\displaystyle \frac{M_\ep^+}{M_\ep^-} \to +\infty$,
\item[(iii)] $|x_\ep| \to 0$,  where $x_\ep \in B_R$ is any point such that $u_\ep(x_\ep)=M_\ep^+$,
\item[(iv)] the rescaled function
\[
\tilde u_{\ep}(x) := \frac{1}{M_{\ep}^+}u_{\ep}\left( \frac{x}{({M_{\ep}^+})^{\frac{2}{n-2s}-\frac{\ep}{2s}}}\right), \quad x \in \R^n,
\]
 converges in $C^{0,\alpha}_{loc}(\R^n)$, for some $\alpha \in (0,1)$, to the fractional standard bubble $U_{0,\mu}$ in $\R^n$ centered at the origin and such that $U_{0,\mu}(0)=1$,
\item[(v)] if $s \in (\frac{1}{2}, 1)$ then $|y_\ep| \to 0$,  where $y_\ep \in B_R$ is any point such that $|u_\ep(y_\ep)|=M_\ep^-$.
\end{itemize}
\end{theorem}

Theorem \ref{mainteo} establishes the first existence result of sign-changing bubble-tower solutions for non-local semilinear elliptic problems driven by the $s$-Laplacian, when $s>\frac{1}{2}$. For $s \in (0,\frac{1}{2}]$ we still get that the positive and the negative part blow-up with different speeds, but for the negative part we cannot provide any information about its concentration point.
From a technical point of view (see the proof of Lemma \ref{L:maxpointbehavior}) this is due to the fractional Strauss inequality for radial functions, namely
\begin{equation}\label{straussineq}
\sup_{x \in \R^n \setminus\{0\}}|x|^{\frac{n-2s}{2}}|u(x)| \leq K_{n,s} \|u\|^2_s,
\end{equation}
 where $K_{n,s}$ is an explicit positive constant depending only on $n$, $s$. Indeed, as pointed out in \cite[Remark 2, Remark 4]{choozawa},  \eqref{straussineq} does not hold when $s \in (0,\frac{1}{2}]$. We also stress that in view of the non-local nature of our problem the positive and negative parts are are not, in general, sub or super solutions to Problem \eqref{fracsemilinear} in their domain of definition, so it seems quite hard to overcome this difficulty by applying scaling arguments to $u_\ep^+, u_\ep^-$ separately.

On the other hand, as proved in \cite{BEP3} for the Laplacian, if the blow-up speeds of $u_\ep^+, u_\ep^-$ are comparable then they must concentrate at two separate points. Therefore, in view of (ii), we believe that also for $s \in (0,\frac{1}{2}]$ the negative part concentrates at the center of the ball. We plan to investigate this question  in separate paper. In addition, we think that, as done in \cite{PistWeth} for the Laplacian, by using a Lyapunov--Schmidt reduction method it should be possible to construct sign-changing bubble-tower solutions in general bounded domains, for all $s\in(0,1)$.

We  point out that, thanks to (ii) and (iii), any global maximum point is close to the origin, when $\ep>0$ is sufficiently small. Moreover, in Lemma \ref{L:maxasynt} we specify that any such a point belongs to the nodal component containing the origin and blows-up faster than any other extremal value achieved in the other nodal components, independently on the number of sign-changes. In the local case, by using ODE techniques, it is well known that the global maximum point is the origin and the absolute values of the extrema are ordered in a radially decreasing way. Our result allows to recover these properties, at least asymptotically, via PDE-only arguments.
\medskip

In the second part of this work we study the nodal set of least energy radial sign-changing solutions to \eqref{fracradsubcrit}. 
We remark that, if $u_\ep$ is a nodal solution  to \eqref{fracradsubcrit} and $u_\ep\geq 0$ in a subdomain $D\subset B_R$, the fractional strong maximum principle does not ensure, in general, that $u_\ep > 0$ in $D$ (see \cite[Remark 4.2]{CabSir} and \cite[Sect. 1]{CorIac}). In addition \cite[Theorem 1.4]{fellifall} only grants that $u_\ep$ does not vanish in a set of positive measure. Nevertheless, combining the results of \cite{CorIac} with a new argument based on energy and regularity estimates, we show that for any $s\in(0,1)$ least energy radial sign-changing solutions to \eqref{fracradsubcrit} vanish only where a change of sign occurs (see Lemma \ref{L:nozeroorigin}, Lemma \ref{L:nodalchar}).

Finally, in Theorem \ref{mainteo2} we prove that for any $s_0 \in (0,1)$, if there exists a $L^2(B_R)$-continuous family $\mathcal{A}=\{u_{\ep,s}\}_{s \in [s_0,1)}$ of least energy nodal radial solutions to \eqref{fracradsubcrit}, then every element of the family changes sign exactly once, provided that $\ep>0$ is small enough.  The key ingredients of the proof are the estimates contained in \cite[Theorem 1.2]{RosSer}, and the continuity of the map $s\mapsto C_{\mathcal{M}^r(B_R)}(s,\ep)$, where $C_{\mathcal{M}^r(B_R)}(s,\ep)$ is the infimum of the energy over the nodal Nehari set, which is a new result of its own interest (see Proposition \ref{P:contenermap}).
\medskip


The outline of the paper is the following: in Section 2 we fix the notation and we recall some known results about the existence of sign-changing solutions to \eqref{fracradsubcrit}, in Section 3 we study the asymptotic behavior, as $\ep\to 0^+$, of the energy levels $C_{\mathcal{M}(\Omega)}(s,\ep)$ in generic bounded domains. In Section 4 we prove Theorem \ref{mainteo}. Finally in Section 5 we analyze the nodal set of least energy radial sign-changing solutions to \eqref{fracradsubcrit} and we prove Theorem \ref{mainteo2}.

\section{Notation and preliminary results}
In this section we recall some definitions and known facts that will be used in this work.
\subsection{Functional setting, standard bubbles}
In this paper $(-\Delta)^s$ stands for the (restricted) $s$-Laplacian operator, which is formally defined as
\[
(-\Delta)^s u(x) := C_{n,s} P.V. \int_{\R^n}\frac{u(x)-u(y)}{|x-y|^{n+2s}}\de y= C_{n,s} \lim_{\varepsilon \to 0^+} \int_{\R^n \setminus B_\varepsilon(x)}\frac{u(x)-u(y)}{|x-y|^{n+2s}}\de y,
\]
where the constant $C_{n,s}$ is given by

\[
C_{n,s} := \frac{2^{2s} \Gamma\left(\frac{n}{2}+s\right)}{\pi^{\frac{n}{2}}|\Gamma(-s)|}.
\]
Let $s \in (0,1)$ and let $n > 2s$. For a given smooth bounded domain $\Omega \subset \R^n$, we consider as a working functional space the Sobolev space 
\[
X_0^s(\Omega) := \{u \in H^s(\R^n) \ ; \ u = 0 \text{ a.e. in }\R^n\setminus \Omega\},
\] endowed with the norm
\[
\|u\|^2_s := {\frac{C_{n,s}}{2}\int_{\R^{2n}} \frac{|u(x)-u(y)|^2}{|x-y|^{n+2s}}\de x \de y},
\] 
and whose associated scalar product is
\[
(u, v)_s := \frac{C_{n,s}}{2}\int_{\R^{2n}} \frac{(u(x)-u(y))(v(x)-v(y))}{|x-y|^{n+2s}}\de x \de y.
\]
The Sobolev space $\D^s(\R^n)$ is defined as the completion of $C_0^\infty(\R^n)$ with respect to the above norm. By the fractional Sobolev embedding theorem it holds that $\D^s(\R^n) \hookrightarrow L^{2_s^*}(\R^n)$ and $X_0^s(\Omega) \hookrightarrow L^{p}(\Omega)$ for all $p \in [1, 2_s^*]$, where $2_s^*=\frac{2n}{n-2s}$. The previous embeddings are continuous, and the second one is compact when $p \in [1, 2_s^*)$.
The best Sobolev constant is characterized as
\[
S_s := \inf_{v \in \D^s(\R^n)\setminus\{0\}}\frac{\|v\|^2_s}{|v|^2_{2^*_s}},
\]
where $|\cdot|_{p}$ denotes the usual $L^{p}$-norm, for $p \in [1, \infty]$. To simplify the notation we will not specify the domain of integration in $|\cdot|_{p}$, but it will be always clear from the context that it is either $\R^n$, or a fixed bounded domain $\Omega$, or a family of bounded domains when considering rescaled functions.
The value of $S_s$ is explicitly known (see \cite{CotTav}), it depends continuously on $s \in [0,1]$, and it is achieved exactly by the family
\[
\mathcal{U}_{\mu, x_0,k}(x):=k\left(\frac{\mu}{\mu^2 + |x-x_0|^2}\right)^\frac{n-2s}{2}, \quad \mu >0,\ x_0 \in \R^n,\ k \in \R. 
\]
If we choose $k = b_{n,s}$, where
\begin{equation}\label{eq:defbns}
b_{n,s}:= 2^{\frac{n-2s}{2}}\left( \frac{\Gamma\left( \frac{n+2s}{2}\right)}{\Gamma\left( \frac{n-2s}{2}\right)}\right)^{\frac{n-2s}{4s}},
\end{equation}
then the functions
\begin{equation}\label{fracbubble}
U_{x_0, \mu}(x) := b_{n,s}\left(\frac{\mu}{\mu^2 + |x-x_0|^2}\right)^\frac{n-2s}{2},
\end{equation}
also known as ``standard fractional bubbles", satisfy
\begin{equation}\label{Eq:Lan-Em}
(-\Delta)^s U_{x_0, \mu} = U_{x_0, \mu}^{2^*_s-1} \quad \text{ in }\R^n
\end{equation}
for all $\mu>0, x_0 \in \R^n$ and
\[
\|U_{x_0, \mu}\|^2_s = |U_{x_0, \mu}|^{2^*_s}_{2^*_s}.
\]

\subsection{Existence of constant-sign and sign-changing solutions}
Let $\Omega \subset \R^n$ be a smooth bounded domain and consider the problem 
\begin{equation}\label{fracsubcrit}
\begin{cases}
(-\Delta)^s u = |u|^{2^*_s-2-\varepsilon}u &\text{in }\Omega, \\
u = 0 &\text{in }\R^n \setminus \Omega,
\end{cases}
\end{equation}
where $\ep \in (0, 2^*_s-2)$.
Weak solutions to \eqref{fracsubcrit} correspond to critical points of the functional
\[
I_{s,\varepsilon}(u) := \frac{1}{2}\|u\|^2_s - \frac{1}{2^*_s-\varepsilon}|u|_{2^*_s-\varepsilon}^{2^*_s-\varepsilon}. 
\]
 The Nehari manifold and the nodal Nehari set are, respectively, defined by
\[
\begin{aligned}
&\mathcal{N}_{s,\varepsilon}(\Omega) : = \{ u \in X_0^s(\Omega) \ ; \ I_{s, \varepsilon}'(u)[u] = 0, u \not \equiv 0\},\\
&\mathcal{M}_{s, \varepsilon}(\Omega) := \{u \in X_0^s(\Omega) \ ; \ I_{s, \varepsilon}'(u)[u^\pm] = 0 , u^\pm \not \equiv 0\}.
\end{aligned}
\]
Since we deal with subcritical nonlinearities, by standard variational methods we know that there exists a minimizer $u_\varepsilon \in \mathcal{N}_{s, \varepsilon}(\Omega)$ of $I_{s, \varepsilon}$, and we set 
\[
C_{\mathcal{N}(\Omega)}(s, \varepsilon):= \inf_{v \in \mathcal{N}_{s, \varepsilon}(\Omega)}I_{s, \varepsilon}(v).
\]
Moreover, the minimizer is a weak solution to \eqref{fracsubcrit} and it is of constant sign. We also remark that, equivalently, constant-sign weak solutions to \eqref{fracsubcrit} can be found as minimizers to
\[
S_{s,\varepsilon} := \inf_{v \in X^s_0(\Omega) \setminus \{0\}}\frac{\|v\|^2_s}{|v|^2_{2^*_s-\varepsilon}},
\]
and the following relation holds
\begin{equation}\label{E:posenequiv}
C_{\mathcal{N}(\Omega)}(s, \varepsilon) = \frac{2^*_s-2-\varepsilon}{2(2^*_s-\varepsilon)}S_{s,\varepsilon}^{\frac{2^*_s-\varepsilon}{2^*_s-2-\varepsilon}}.
\end{equation}

In the case of sign-changing solutions, as proved in \cite{TWW}, there exists a minimizer of the energy over the nodal Nehari set, and it is a weak solution to \eqref{fracsubcrit}. We refer to such solutions as least energy sign-changing (or nodal) solutions and we set
\[
C_{\mathcal{M}(\Omega)}(s, \varepsilon) := \inf_{v \in \mathcal{M}_{s, \varepsilon}(\Omega)}I_{s, \varepsilon}(v).
\]

Let us now turn our attention to the radial case. Taking $\Omega = B_R$, where $B_R= B_R(0)$ denotes the ball in $\R^n$ of radius $R>0$ centered at the origin, we set
\[
\begin{aligned}
&\mathcal{N}^r_{s,\varepsilon}(B_R): = \{ u \in X^s_0(B_R) \ ; \ u \in \mathcal{N}_{s,\varepsilon}(B_R) \text{ and } u \text{ is radially symmetric}\},\\
&\mathcal{M}^r_{s, \varepsilon} (B_R) := \{u \in X^s_0(B_R) \ ; \ u \in \mathcal{M}_{s,\varepsilon}(B_R)\text{ and } u \text{ is radially symmetric}\}.
\end{aligned}
\]
As a consequence of the fractional moving plane method (see \cite{ChenLiLi}), positive solutions of \eqref{fracsubcrit} in $B_R$ are radially symmetric and radially decreasing. In particular, it holds that
\[
C_{\mathcal{N}(B_R)}(s, \varepsilon) = C_{\mathcal{N}^{r}(B_R)}(s, \varepsilon) := \inf_{v \in \mathcal{N}^r_{s, \varepsilon}(B_R)}I_{s, \varepsilon}(v).
\]
Concerning the case of nodal solutions, arguing as in \cite{TWW} we obtain least energy radial sign-changing solutions as minimizers of the energy over the radial nodal Nehari set, and as before we denote
\[ 
C_{\mathcal{M}^{r}(B_R)}(s, \varepsilon) := \inf_{v \in \mathcal{M}^r_{s, \varepsilon}(B_R)}I_{s, \varepsilon}(v).
\]
We point out that it is not known whether or not $C_{\mathcal{M}^{r}(B_R)}(s, \varepsilon)$ coincide with $C_{\mathcal{M}(B_R)}(s, \varepsilon)$, but they have the same limit when $\ep\to 0^+$ (see Lemma \ref{L:nodalenerasynt}).

\section{Asymptotic analysis of the energy levels as $\varepsilon\to 0^+$}

In this section we study the asymptotic behavior  as $\ep\to 0^+$ of the energy levels $C_{\mathcal{N}(\Omega)}(s, \varepsilon)$, $C_{\mathcal{M}(\Omega)}(s, \varepsilon)$ defined in Sect. 2. We begin with the following technical result. 

\begin{lemma}\label{bubbleconc}
Let $s \in (0,1)$ and $n>2s$. Let $\Omega\subset \R^n$ be a domain, let $x_0 \in \Omega$ and $\rho>0$ be such that $B_{4\rho}(x_0) \subset \Omega$. Let $\varphi \in C^{\infty}_c(\Omega)$ be such that $supp(\varphi) \subset B_{2\rho}(x_0)$, $0 \leq \varphi \leq 1$ in $B_{2\rho}(x_0)$ and $\varphi \equiv 1$ in $B_\rho(x_0)$. There exists $\tau_0 >0$ such that for every $\tau \in (0,\tau_0)$, setting
\begin{equation}\label{Eq:restfracbub}
u^s_\tau(x):= \varphi (x) \tau^{-\left(\frac{n-2s}{2}\right)} U_{x_0, \mu}\left( \frac{x-x_0}{\tau} + x_0 \right),
\end{equation}
where $U_{x_0, \mu}$ is defined by \eqref{fracbubble}, then the following estimates hold: 
\begin{equation}\label{critapproxtalent}
\begin{aligned}
&\|u^s_\tau\|^2_s \leq S_s^{\frac{n}{2s}} + C\tau^{n-2s},\\
&0 < S_s^{\frac{n}{2s}} - C\tau^n \leq |u^s_\tau|^{2^*_s}_{2^*_s}\leq S_s^{\frac{n}{2s}},\\
&0 \leq |u^s_\tau|_1 \leq C\tau^{\frac{n-2s}{2}},
\end{aligned}
\end{equation}
where the constants $C$ are positive and depend only on $n$, $s$, $x_0$, $\mu$ and $\rho$.
Moreover, for any $0<\varepsilon < \frac{2s}{n-2s}$, taking $\mu = b_{n,s}^{\frac{2}{n-2s}}$, where $b_{n,s}$ is given by \eqref{eq:defbns}, we have
\begin{equation}\label{subcritapproxtalent}
\begin{aligned}
&0<\tau^{\left(\frac{n-2s}{2}\right)\varepsilon}\left[S_s^{\frac{n}{2s}}- C\tau^n\right] \leq |u^s_\tau|^{2^*_s-\varepsilon}_{2^*_s-\varepsilon}\leq C \tau^{\left(\frac{n-2s}{2}\right)\varepsilon},\\
&0< \tau^{\left(\frac{n-2s}{2}\right)(1+\varepsilon)}\left[S_s^{\frac{n}{2s}}- C\tau^n\right] \leq |u^s_\tau|^{2^*_s-1-\varepsilon}_{2^*_s-1-\varepsilon}\leq C \tau^{\left(\frac{n-2s}{2}\right)(1+\varepsilon)},
\end{aligned}
\end{equation}
where the appearing constants are positive and depend only on $n$, $s$, $x_0$ and $\rho$.
Let $0<s_0< s_1\leq 1$ and let $n > 2s_1$. Then, if $s \in [s_0, s_1)$ and $\varepsilon \in \left(0, \frac{2s_0}{n-2s_0}\right)$, both $\tau_0$ and the above constants $C$ can be taken in such a way that they depend on $n$, $\mu$, $\rho$, $s_0$, $s_1$, but not on $s$, $\tau$ and $\ep$. 
\end{lemma}
\begin{proof}
Inequalities \eqref{critapproxtalent} are proved in \cite{Ser}, \cite{SerVal} and hold true for all sufficiently small $\tau>0$ with constants $C$ independent on $\tau$. Concerning the dependence of the constants $C$ on the other parameters we refer to \cite[Remark 2.2]{CorIac}. Let us focus on the proof of \eqref{subcritapproxtalent}. Taking if necessary a smaller $\tau_0>0$ so that $\tau_0 < \min\{1, 2\rho\}$, we find that, when $q =2^*_s-\varepsilon$ or $q=2^*_s-1-\varepsilon$,
\[
\int_{\R^n}|u^s_\tau|^q \de x \leq C \tau^{n - \left(\frac{n-2s}{2}\right)q}\left(C+ \int_1^{\frac{2\rho}{\tau}}r^{n- (n-2s)q-1} \de r \right), 
\]
where the constants $C>0$ depend on $n, s, \mu$, but not on $\tau$ nor on $\varepsilon$. Furthermore, since $0<\varepsilon<\frac{2s}{n-2s}$ we have
\[
\int_1^{\frac{2\rho}{\tau}}r^{n- (n-2s)q-1} \de r \leq C,
\]
for some constant $C>0$ independent on $\tau$ and $\varepsilon$. Recalling the definition of $b_{n,s}$, one can see that all the previous constants can be taken in a uniform way with respect to $s \in [s_0, s_1)$ when $n > 2s_1$ and $\varepsilon \in \left(0, \frac{2s_0}{n-2s_0}\right)$. Hence the right-hand side inequalities in \eqref{subcritapproxtalent} are proved.
 
In order to prove the the left-hand side inequalities it suffice to notice that, thanks to our choice of $\mu= b_{n,s}^{\frac{2}{n-2s}}$, it follows that $|U_{\mu,x_0}^s|_\infty=1$ and thus $|U^s_{ \mu, x_0}|^q \geq |U^s_{ \mu, x_0}|^{2^*_s}$ for every $x \in \R^n$, where $q =2^*_s-\varepsilon$ or $q=2^*_s-1-\varepsilon$. Then, using also \eqref{critapproxtalent}, we find that
\[
\begin{aligned}
\int_{\Omega} |u^s_\tau|^{q} \de x &\geq \tau^{n - \left(\frac{n-2s}{2}\right)q}\int_{B_{\rho/\tau}}|U^s_{\mu,x_0}|^q \de x\\
&\geq \tau^{n - \left(\frac{n-2s}{2}\right)q}\left[S_s^{\frac{n}{2s}}- \int_{\R^n \setminus B_{\rho/\tau}}|U^s_{\mu, x_0}|^{2^*_s}\de x\right]\\
&\geq \tau^{n - \left(\frac{n-2s}{2}\right)q}\left[S_s^{\frac{n}{2s}}- C\tau^n\right],
\end{aligned}
\]
for some constant $C >0$ which depends only on $n$, $s$ and $\rho$, but not on $\tau$, $\ep$, and which is uniform with respect to $s \in [s_0, s_1)$. The proof is complete.
\end{proof}

As a consequence we obtain the following uniform asymptotic result on $C_{\mathcal{N}(\Omega)}(s,\varepsilon)$. 

\begin{lemma}\label{AsintPos}
Let $s \in (0,1)$, $n>2s$ and let $\Omega$ be a smooth bounded domain of $\R^n$. Then, as $\varepsilon \to 0^+$, it holds 
\begin{equation}\label{asyntener1}
C_{\mathcal{N}(\Omega)}(s,\varepsilon) \to \frac{s}{n}S_s^{\frac{n}{2s}}.
\end{equation}
Moreover, if $0<s_0< s_1\leq 1$ and $n>2s_1$, for every $\varepsilon \in  \left(0, \min\left\{\tau_0, \frac{2s_0}{n-2s_0},1\right\}\right)$, where $\tau_0$ is given by Lemma \ref{bubbleconc}, we have
\begin{equation}\label{asyntener1uni}
\sup_{s \in (s_0,s_1)}\left|C_{\mathcal{N}(\Omega)}(s,\varepsilon) - \frac{s}{n}S_s^{\frac{n}{2s}}\right|\leq g_1(\varepsilon),  
\end{equation}
where $g_1$ does not depend on $s$ and $g_1(\varepsilon) \to 0^+$ as $\varepsilon \to 0^+$. 
\end{lemma}
\begin{proof}
Let $s \in (0,1)$ and $n>2s$. In order to prove \eqref{asyntener1}, in view of \eqref{E:posenequiv}, it is sufficient to show that $S_{s,\varepsilon} \to S_s$, as $\ep\to 0^+$, where $S_{s,\varepsilon} := \inf_{u \in X^s_0(\Omega) \setminus \{0\}} \frac{\|u\|^2_s}{|u|^2_{2^*_s-\varepsilon}}$. 
To this end we observe that, by H\"older's inequality, for every $u \in X^s_0(\Omega)$ and any sufficiently small $\ep>0$ we have
\[
|u|_{2^*_s-\varepsilon} \leq |u|_2^{\frac{2\varepsilon}{(2^*_s-\varepsilon)(2^*_s-2)}} |u|_{2^*_s}^{\frac{2^*_s(2^*_s-2-\varepsilon)}{(2^*_s-\varepsilon)(2^*_s-2)}}.
\]
Then, thanks to the fractional Sobolev embedding and the variational characterization of the eigenvalues, we infer that
\[
|u|^2_{2^*_s-\varepsilon} \leq \lambda_{1,s}(\Omega)^{-\frac{2\varepsilon}{(2^*_s-\varepsilon)(2^*_s-2)}} S_s^{-\frac{2^*_s(2^*_s-2-\varepsilon)}{(2^*_s-\varepsilon)(2^*_s-2)}}\|u\|_s^2, 
\]
which implies, when $u \not \equiv 0$, that
\[
\lambda_{1,s}(\Omega)^{\frac{2\varepsilon}{(2^*_s-\varepsilon)(2^*_s-2)}} S_s^{\frac{2^*_s(2^*_s-2-\varepsilon)}{(2^*_s-\varepsilon)(2^*_s-2)}} \leq \frac{\|u\|^2_s}{|u|^2_{2^*_s-\varepsilon}}.
\]
Taking the infimum as $u \in X^s_0(\Omega) \setminus \{0\}$ we get that
\begin{equation}\label{eq:firsteigenestimate}
\lambda_{1,s}(\Omega)^{\frac{2\varepsilon}{(2^*_s-\varepsilon)(2^*_s-2)}} S_s^{\frac{2^*_s(2^*_s-2-\varepsilon)}{(2^*_s-\varepsilon)(2^*_s-2)}} \leq S_{s,\varepsilon},
\end{equation}
and thus it follows that
\begin{equation}\label{passag665}
S_s \leq \liminf_{\varepsilon \to 0^+}S_{s, \varepsilon}.
\end{equation}
Now, let us fix $x_0\in \Omega$, $\rho>0$, $\varphi$ as in the statement of Lemma \ref{bubbleconc} and take $\mu= b_{n,s}^{\frac{2}{n-2s}}$. Let $u^s_\tau$ be the function defined in \eqref{Eq:restfracbub}. Using both \eqref{critapproxtalent} and \eqref{subcritapproxtalent}, then for any $\tau \in (0, \tau_0)$, $\ep \in (0,\frac{2s}{n-2s})$ we obtain
\begin{equation}\label{pass666}
S_{s,\varepsilon} \leq \frac{\|u^s_{\tau}\|^2_s}{|u^s_\tau|_{2^*_s-\varepsilon}^2} \leq \frac{S_s^{\frac{n}{2s}} + C\tau^{n-2s}}{\tau^{\left(\frac{n-2s}{2^*_s-\varepsilon}\right)\varepsilon}\left[S_s^{\frac{n}{2s}}- C\tau^n\right]^\frac{2}{2^*_s-\varepsilon}}.
\end{equation}
Hence, for $\ep \in (0,\min\{\tau_0, \frac{2s}{n-2s}\})$, taking $\tau = \varepsilon$ in \eqref{pass666} and by elementary computations, we infer that
\[
\limsup_{\varepsilon \to 0^+}S_{s,\varepsilon} \leq S_s, 
\]
which, together with \eqref{passag665}, implies \eqref{asyntener1}. The first part of the Lemma is thus proved.

For the second part, recalling \cite[(2.5), (2.8)]{CorIac} we have that, fixing $0<s_0< 1$, there exist two positive constants $\underline\lambda$, $\overline\lambda$ such that
\begin{equation}\label{constantbound}
\begin{aligned}
&\underline\lambda \leq \lambda_{1,s}(\Omega)\leq \overline \lambda & &\forall s \in [s_0,1).
\end{aligned}
\end{equation}
Hence, from \eqref{eq:firsteigenestimate}, \eqref{constantbound} we deduce that there exists $C>0$ depending only on $n$, $\Omega$ and $s_0$ such that for all $s \in [s_0, 1)$
\begin{equation}\label{passag15}
S_{s,\varepsilon}\geq S_s C^{\varepsilon}.
\end{equation}
On the other hand, let us fix $s_1$ such that $0<s_0<s_1\leq 1$ and let $s \in (s_0, s_1)$, $n >2s_1$, $\ep \in \left(0, \min\left\{\tau_0, \frac{2s_0}{n-2s_0},1\right\}\right)$, where $\tau_0$ is given by Lemma \ref{bubbleconc}. Then from \eqref{pass666}, \eqref{constantbound}, choosing $\tau=\ep$ and taking into account that $\frac{1}{2}<\varepsilon^\varepsilon \leq 1$ for any $\ep \in (0,1)$, we deduce that
\begin{equation}\label{passag16}
S_{s, \varepsilon} \leq S_s\left(\frac{C^\varepsilon}{(\varepsilon^\varepsilon)^{\alpha}[1-C\varepsilon^n]^{\beta}}\right) + C\varepsilon^{n-2s},
\end{equation}
for some constants $C, \alpha, \beta >0$ which depend only on $n$, $s_0$ and $s_1$, but not on $s$ and $\varepsilon$. 
Therefore, from \eqref{E:posenequiv}, \eqref{passag15} and \eqref{passag16} we obtain 
\[
C^\varepsilon \frac{s}{n}S_s^{\frac{n}{2s}} \leq C_{\mathcal{N}(\Omega)}(s, \varepsilon)\leq \frac{s}{n}S_s^{\frac{n}{2s}}g(\varepsilon), 
\]
where $g$ and $C>0$ do not depend on $s$, and $g$ is such that $g(\varepsilon) \to 1$ as $\varepsilon \to 0^+$. Hence, setting $g_1(\ep):=\max\{|C^\varepsilon-1|, |g(\ep)-1|\}$ we get \eqref{asyntener1uni}. The proof is then complete.
\end{proof}

In the next result we describe the asymptotic behavior of $C_{\mathcal{M}(\Omega)}(s, \varepsilon)$, as $\ep\to 0^+$. 
Differently from the case of critical nonlinearities (see \cite[Lemma 3.6]{CorIac}), there are some difficulties in proving uniform energy estimates from above which are directly related to $C_{\mathcal{N}(\Omega)}(s, \varepsilon)$. To overcome these difficulties we provide a uniform upper bound in terms of $\frac{2s}{n}S_s^{\frac{n}{2s}}$ instead, which is obtained by using as competitors for the energy superpositions of standard bubbles centered at the same point and with different concentration speeds.

\begin{lemma}\label{L:nodalenerasynt}
Let $s \in (0,1)$, $n>2s$ and let $\Omega \subset \R^n$ be a smooth bounded domain. We have
\begin{equation}\label{eq:asrelneharinodal}
\lim_{\varepsilon \to 0^+} C_{{\mathcal{M}}(\Omega)}(s, \varepsilon) = \frac{2s}{n}S_s^{\frac{n}{2s}}.
\end{equation}
Moreover, let $0<s_0<s_1 \leq1$ and $n >2s_1$. Then there exists $\hat \varepsilon=\hat \varepsilon(s_0, s_1) \in (0, 2^*_{s_0}-2)$ such that for every $\varepsilon \in (0, \hat \varepsilon)$
\begin{equation}\label{Eq:uniformen}
\sup_{s \in (s_0, s_1)}\left|C_{\mathcal{M}(\Omega)}(s, \varepsilon) - \frac{2s}{n}S_s^{\frac{n}{2s}}\right|\leq g_2(\varepsilon), 
\end{equation}
where the function $g_2$ does not depend on $s$ and $g_2(\varepsilon) \to 0$ as $\varepsilon \to 0^+$.
The same result holds for $\mathcal{M}^r_{s, \varepsilon}(B_R)$. 

\end{lemma}
\begin{proof}
Let us fix $s \in (0,1)$, $n>2s$ and let $\Omega \subset \R^n$ be a smooth bounded domain. We claim that
\begin{equation}\label{liminfineq}
2C_{\mathcal{N}(\Omega)}(s, \varepsilon) \leq C_{\mathcal{M}(\Omega)}(s, \varepsilon).
\end{equation}
As an immediate consequence, from Lemma \ref{AsintPos}, we get that
\begin{equation}\label{eq:liminfasrelneharinodal}
\frac{2s}{n}S_s^{\frac{n}{2s}} \leq \liminf_{\varepsilon \to 0}C_{\mathcal{M}(\Omega)}(s, \varepsilon).
\end{equation}

To prove \eqref{liminfineq} it suffices to notice that, given $u \in \mathcal{M}_{s, \varepsilon}(\Omega)$, then for every $\alpha, \beta >0$ it holds
\[
I_{s,\varepsilon}(\alpha u^+) + I_{s, \varepsilon}(\beta u^-) \leq I_{s, \varepsilon}(u). 
\]
This follows from the explicit computation of $I_{s, \varepsilon}(\alpha u^+ - \beta u^-)$, taking into account that $(u^+, u^-)_s <0$ and that \(\sup_{t \geq 0}\left(\frac{t^2}{2} - \frac{t^{2^*_s-\varepsilon}}{2^*_s-\varepsilon}\right)\leq \left(\frac{1}{2} - \frac{1}{2^*_s-\varepsilon}\right).
\)
Hence, choosing $u\in \mathcal{M}_{s, \varepsilon}(\Omega)$ such that $I_{s, \varepsilon}(u) = C_{\mathcal{M}(\Omega)}(s, \varepsilon)$ and $\alpha, \beta $ in such a way that $\alpha u^+, \beta u^- \in \mathcal{N}_{s, \varepsilon}(\Omega)$ (which is always possible), we obtain the desired result.
\medskip

 To conclude the proof of \eqref{eq:asrelneharinodal} we need to prove the $\limsup$ inequality. To this end, we consider $u^s_{\tau'}$ and ${u^s_{\tau''}}$ of the form \eqref{Eq:restfracbub}, sharing all the parameters $\mu$, $\varphi$, $\rho$, $x_0$, apart from $\tau$. To simplify the notation, we assume without loss of generality that $0 \in \Omega$ and we take $x_0=0$. Moreover, we choose $\rho$ and $\mu$ as in Lemma \ref{bubbleconc} so that \eqref{critapproxtalent}, \eqref{subcritapproxtalent} hold true whenever $\ep$ is small enough. Finally, for the concentration parameters, we take $\tau', \tau''$ of the form $\tau' = \varepsilon^{\frac{2\delta}{n-2s}}$, $\tau''= \varepsilon^{\frac{2}{n-2s}}$, where $\delta >0 $ is such that 
\[
\delta > \max\left\{1, \frac{(2^*_s-\varepsilon)(2^*_s-1-\varepsilon)}{2^*_s-2\varepsilon}, 2^*_s-1\right\}.
\]
Notice that $\delta=2^*_s-1$, when $\ep$ is small enough, and that it can be taken in a uniform way with respect to $s$ when $s \in [s_0, s_1)$. 

Arguing as in \cite[Theorem 3.5, Step 2]{CorIac}, we infer that 
\begin{equation}\label{stima23}
C_{\mathcal{M}(\Omega)}(s, \varepsilon) \leq \sup_{\alpha, \beta \geq 0 }I_{s, \varepsilon}(\alpha u^s_{\tau'} - \beta u^s_{\tau''}). 
\end{equation}
To conclude we need to estimate the right-hand side of \eqref{stima23}. The first crucial fact is that, in \eqref{stima23}, it is sufficient to consider only linear combinations $\alpha u^s_{\tau'} - \beta u^s_{\tau''}$ with $\alpha, \beta$ in a compact subset of $\R^+\cup \{0\}$.
More precisely, we prove that there exists $\tilde C >0$ independent on $\varepsilon$ (and depending only on $s_0, s_1$ when $s \in [s_0, s_1))$ such that, for any $\alpha, \beta \geq 0$ satisfying  $\alpha + \beta \geq \tilde C$, it holds
\begin{equation}\label{E:firsthalfsup}
I_{s, \varepsilon}(\alpha u^s_{\tau'} - \beta u^s_{\tau''}) \leq 0.
\end{equation}
Indeed, by a straightforward computation and using Lemma \ref{bubbleconc} we have
\begin{equation}\label{paramext1}
\|\alpha u^s_{\tau'} - \beta u^s_{\tau''}\|^2_s \leq C(\alpha + \beta)^2, 
\end{equation}
for some constant $C$ independent on both $\varepsilon$, $\tau'$, $\tau''$ and $s$, when $s \in [s_0,s_1)$. 
On the other hand, arguing exactly as in \cite[Lemma 3.6]{CorIac} and using again Lemma \ref{bubbleconc}, we infer that for any $\theta \in(0,1)$ 
\begin{equation}\label{eq:stimacontheta}
\begin{aligned}
|\alpha u^s_{\tau'}& - \beta u^s_{\tau''}|^{2^*_s-\varepsilon}_{2^*_s-\varepsilon}\\
&\geq C\alpha^{2^*_s-\varepsilon} (\tau')^{\left(\frac{n-2s}{2}\right)\varepsilon}\left(C - \frac{(\tau')^{\left(\frac{n-2s}{2}\right)(2^*_s-2\varepsilon)}}{\theta^{2^*_s-1-\varepsilon}} - \frac{(\tau')^{\left(\frac{n-2s}{2}\right)\frac{2^*_s}{2^*_s-\varepsilon -1}}}{\theta^{\frac{1}{2^*_s-1-\varepsilon}}}\right)\\
&+ \beta^{2^*_s-\varepsilon}( |u^s_{\tau''}|^{2^*_s-\varepsilon}_{2^*_s-\varepsilon} - \theta |u^s_{\tau''}|_\infty^{2^*_s-\varepsilon}).
\end{aligned}
\end{equation}
Now, thanks to our choice of $\mu$ we have 
\[
|u^s_{\tau''}|_\infty =u^s_{\tau''}(0) = (\tau'')^{-\frac{n-2s}{2}}.
\]
Hence, recalling that $\tau'' = \varepsilon^{\frac{2}{n-2s}}$ and taking $\theta = C'\varepsilon^{{2^*_s-\varepsilon}}$, where $C'$ will be chosen later, from Lemma \ref{bubbleconc} we obtain that
\[
|u^s_{\tau''}|^{2^*_s-\varepsilon}_{2^*_s-\varepsilon} - \theta |u^s_{\tau''}|_\infty^{2^*_s-\varepsilon} \geq C- C',
\]
 for any $\varepsilon>0$ small enough, where $C$ does not depend on $\varepsilon$, nor on $s$ when $s \in [s_0, s_1)$. Therefore, taking $C'=\frac{1}{2} C$ we get that
 \begin{equation}\label{eq:stimacontheta2}
|u^s_{\tau''}|^{2^*_s-\varepsilon}_{2^*_s-\varepsilon} - \theta |u^s_{\tau''}|_\infty^{2^*_s-\varepsilon} \geq \frac{1}{2}C>0.
\end{equation}
Thus, recalling that $\tau' = \varepsilon^\frac{2\delta}{n-2s}$, from \eqref{eq:stimacontheta}, \eqref{eq:stimacontheta2} we obtain
\[
\begin{aligned}
&|\alpha u^s_{\tau'} - \beta u^s_{\tau''}|^{2^*_s-\varepsilon}_{2^*_s-\varepsilon}\\
& \geq (\varepsilon^{\varepsilon})^{\delta}C\alpha^{2^*_s-\varepsilon}\left(C - \varepsilon^{\delta (2^*_s-2\varepsilon)-(2^*_s-\varepsilon)(2^*_s-1-\varepsilon)}-\varepsilon^{\frac{2^*_s(\delta-1)+\varepsilon}{2^*_s-1-\varepsilon}}\right)+  \frac{1}{2}C \beta^{2^*_s-\varepsilon}.
\end{aligned}
\]
Then, exploiting the properties of the function $t \mapsto t^t$ and thanks to the definition of $\delta$, we find $C>0$ such that for all sufficiently small $\ep>0$
\begin{equation}\label{paramext2}
|\alpha u^s_{\tau'} - \beta u^s_{\tau''}|^{2^*_s-\varepsilon}_{2^*_s-\varepsilon}\geq C(\alpha^{2^*_s-\varepsilon}+ \beta^{2^*_s-\varepsilon})\geq C(\alpha + \beta)^{2^*_s-\varepsilon}.
\end{equation}
Finally, thanks to \eqref{paramext1} and \eqref{paramext2} we infer that
\[
I_{s, \varepsilon}(\alpha u^s_{\tau'} - \beta u_{\tau''}^s) \leq C(\alpha + \beta)^2(1 -  C(\alpha+ \beta)^{2^*_s-2-\varepsilon}),
\]
which implies that there exists $\tilde C >0$, not depending on $\varepsilon$, such that if $(\alpha + \beta) \geq \tilde C$ then $I_{s, \varepsilon}(\alpha u^s_{\tau'} - \beta u_{\tau''}^s) \leq 0$, as claimed. We observe that $\tilde C$ can be taken in a uniform way with respect to $s$, when $s \in [s_0,s_1)$.
\medskip

It remains to treat the case $\alpha + \beta \leq \tilde C$. To this end we begin with a preliminary estimate on the scalar product between two bubbles. A careful analysis of the argument carried out in \cite[Proposition 21]{SerVal} shows that 
\[
|(u^s_{\tau'}, u^s_{\tau''})_s| \leq  (\tau')^{-\frac{n-2s}{2}}(\tau'')^{-\frac{n-2s}{2}}\left| \left( U^s_{0, \mu}\left( \frac{x}{\tau'}\right), U^s_{0, \mu}\left( \frac{x}{\tau''} \right)\right)_s \right|+ C (\tau')^{\frac{n-2s}{2}}(\tau'')^{\frac{n-2s}{2}},
\]
where the constant $C$ does not depend on $\tau'$ nor on $\tau''$, and it is uniformly bounded with respect to $s \in [s_0,s_1)$. 
Performing a change of variables, and recalling that $U^s_{0, \mu}$ solves \eqref{Eq:Lan-Em}, we get that 
\[
\begin{aligned}
\left|\left( U^s_{0, \mu}\left( \frac{x}{\tau'}\right), U^s_{0, \mu}\left( \frac{x}{\tau''} \right)\right)_s \right|&= \left|(\tau')^{n-2s}\left( U^s_{0, \mu}, U^s_{0, \mu}\left( \frac{\tau'}{\tau''} x \right)\right)_s \right| \\
&\leq (\tau')^{n-2s}\int_{\R^n}|U^s_{0, \mu}|^{2^*_s-1}\left|U^s_{0, \mu}\left( \frac{\tau'}{\tau''} x \right)\right|\de x\\
&\leq C(\tau')^{n-2s}\int_{\R^n}(\mu^2 + |x|^2)^{-\frac{n+2s}{2}}\de x \leq C (\tau')^{n-2s},
\end{aligned}
\]
where we used that $|U_{0, \mu}^s|_\infty = 1$, in view of our choice of $\mu$, and where the constant $C>0$ does not depend on $\tau'$ nor on $\tau''$ and it is uniformly bounded with respect to $s \in [s_0, s_1)$. 
Summing up, and recalling the definition of $\tau'$ and $\tau''$, we obtain
\begin{equation}\label{Eq:scalarestimate}
|(u^s_{\tau'}, u^s_{\tau''})_s|\leq C(\varepsilon^{\delta-1} + \varepsilon^{\delta +1}) \leq C\varepsilon^{\delta-1}. 
\end{equation}

Let us finally consider the case of $\alpha + \beta \leq \tilde C$. Arguing as in \cite[Lemma 3.6]{CorIac} and applying Lemma \ref{bubbleconc} and \eqref{Eq:scalarestimate}, we have
\[
\begin{aligned}
I_{s, \varepsilon}(\alpha u^s_{\tau'} - \beta u_{\tau''}^s) &\leq \frac{\alpha^2}{2}\|u^s_{\tau'}\|^2_s + \frac{\beta^2}{2}\|u^s_{\tau''}\|^2_s  - \frac{\alpha^{2^*_s-\varepsilon}}{2^*_s-\varepsilon}|u^s_{\tau'}|^{2^*_s-\varepsilon}_{2^*_s-\varepsilon}- \frac{\beta^{2^*_s-\varepsilon}}{2^*_s-\varepsilon}|u^s_{\tau''}|^{2^*_s-\varepsilon}_{2^*_s-\varepsilon}\\
&+C\int_{\R^n}|u^s_{\tau'}|^{2^*_s-1-\varepsilon}|u^s_{\tau''}|\de x + C\int_{\R^n}|u^s_{\tau''}|^{2^*_s-1-\varepsilon}|u^s_{\tau'}|\de x + C\varepsilon^{\delta-1}\\
&\leq \frac{\alpha^2}{2}S_s^{\frac{n}{2s}} + \frac{\beta^2}{2}S_s^{\frac{n}{2s}}  - \frac{\alpha^{2^*_s-\varepsilon}}{2^*_s-\varepsilon}(\varepsilon^\varepsilon)^\delta(S_s^{\frac{n}{2s}}- C\varepsilon^{2^*_s\delta})\\
&- \frac{\beta^{2^*_s-\varepsilon}}{2^*_s-\varepsilon}(\varepsilon^\varepsilon)(S_s^{\frac{n}{2s}}- C\varepsilon^{2^*_s}) +C\varepsilon^{-1}\varepsilon^{\delta(1+\varepsilon)} + C\varepsilon^{-(2^*_s-1-\varepsilon)}\varepsilon^{\delta}\\
 &+ C\varepsilon^{2\delta} + C\varepsilon^2 + C\varepsilon^{\delta-1},
\end{aligned}
\]
where we used that $|u^s_{\tau''}|_\infty = \varepsilon^{-1}$. Even in this case all the appearing constants are independent on $\varepsilon$, and they are uniformly bounded with respect to $s$ when $s \in [s_0, s_1)$. Then, using again the elementary estimate \(\sup_{t \geq 0}\left(\frac{t^2}{2} - \frac{t^{2^*_s-\varepsilon}}{2^*_s-\varepsilon}\right)\leq \left(\frac{1}{2} - \frac{1}{2^*_s-\varepsilon}\right)\), we get
\begin{equation}\label{E:sechalfsup}
\begin{aligned}
I_{s, \varepsilon}(\alpha u^s_{\tau'} - \beta u_{\tau''}^s) &\leq \frac{2^*_s-2-\varepsilon}{2^*_s-\varepsilon}S_s^{\frac{n}{2s}} + C(1- (\varepsilon^\varepsilon)^\delta) + C(1- \varepsilon^\varepsilon) +  C\varepsilon^{2^*_s\delta}+ C\varepsilon^{2^*_s} \\
&+C\varepsilon^{\delta-1+ \delta \varepsilon}+ C\varepsilon^{\delta-(2^*_s-1-\varepsilon)} + C\varepsilon^{2\delta} + C\varepsilon^2 + C\varepsilon^{\delta-1}\\
&=: \frac{2s}{n}S^{\frac{n}{2s}}_s + g(\varepsilon),
\end{aligned}
\end{equation}
where all the constants $C>0$, and thus $g$, do not depend on $s$, when $s \in [s_0, s_1)$.  In particular, $g$ satisfies $g(\varepsilon) \to 0 $ as $\varepsilon \to 0^+$.

At the end, putting together \eqref{stima23}, \eqref{E:sechalfsup}, taking into account \eqref{E:firsthalfsup}, we obtain
\begin{equation}\label{limsupineq}
C_{\mathcal{M}(\Omega)}(s, \varepsilon) \leq \frac{2s}{n}S_s^{\frac{n}{2s}} + g(\varepsilon), 
\end{equation}
and thus we get that
\[
\limsup_{\varepsilon \to 0^+}C_{\mathcal{M}(\Omega)}(s, \varepsilon) \leq \frac{2s}{n}S_s^{\frac{n}{2s}},
\]
which, together with \eqref{eq:liminfasrelneharinodal}, gives \eqref{eq:asrelneharinodal}.
\medskip

For the proof of the second part, fixing $0<s_0<s_1\leq 1$, then, thanks to Lemma \ref{AsintPos} and the definition of $g$, we deduce that inequalities \eqref{liminfineq} and \eqref{limsupineq} are uniform with respect to $s$ when $s \in [s_0, s_1)$. At the end, arguing as in Lemma \ref{AsintPos} we obtain \eqref{Eq:uniformen}, for some function $g_2$ independent on $s$ and such that $g_2(\ep) \to 0$, as $\ep\to 0^+$.
 
In the radial case the proof is identical. Indeed, since in the construction we take standard bubbles centered at the same point, then the functions $\alpha u^s_{\tau'} - \beta u^s_{\tau''}$ are radial and thus admissible competitors. The proof is then complete. 
\end{proof}

\section{Asymptotic analysis of least energy radial sign-changing solutions}
In this section we study the asymptotic behavior of least energy radial nodal solutions to \eqref{fracradsubcrit}, as $\varepsilon \to 0^+$. Theorem \ref{mainteo} will be a consequence of the results contained in this section. We begin by a couple of preliminary known results.

\begin{lemma}\label{As:energas}
Let $s \in (0,1)$, let $n > 2s$ and let $\Omega \subset \R^n$ be a smooth bounded domain. Let $(u_{s, \varepsilon}) \subset \mathcal{M}_{s, \varepsilon}(\Omega)$ be a family of solutions of Problem \eqref{fracsubcrit} such that $I_{s,\varepsilon}(u_{s,\varepsilon}) = C_{\mathcal{M}(\Omega)}(s, \varepsilon)$ and set $M_{s, \varepsilon, \pm} := |u^\pm_{s, \varepsilon}|_\infty$. As $\varepsilon \to 0^+$ we have:
\begin{enumerate}[(i)]
\item $\|u_{s,\varepsilon}^\pm\|_s^2 \to S_s^{\frac{n}{2s}} $;
\item $|u_{s,\varepsilon}^\pm|_{2_s^*-\varepsilon}^{2_s^*-\varepsilon} \to S_s^{\frac{n}{2s}} $;
\item $(u_{s, \varepsilon}^+, u_{s, \varepsilon}^-)_s \to 0$;
\item $u_{s,\varepsilon} \rightharpoonup 0$ in $X^s_0(\Omega)$;
\item $M_{s, \varepsilon,\pm} \to + \infty$. 
\end{enumerate}
The same results hold for a family $(u_{s, \varepsilon}) \subset \mathcal{M}^r_{s, \varepsilon}(B_R)$ of radial solutions to Problem \eqref{fracradsubcrit} such that $I_{s, \varepsilon}(u_{s,\varepsilon}) = C_{\mathcal{M}^r(B_R)}(s, \varepsilon)$. Moreover, for every $0<s_0< s_1\leq 1$ and $n>2s_1$, the limits $(i)-(iii)$ are uniform with respect to $s \in [s_0,s_1)$. 
\end{lemma}
\begin{proof}
It suffices to argue as in \cite[Lemma 4.3]{CorIac}, with some minor modifications.
\end{proof}

The following estimate will play a central role in this paper.
\begin{proposition}\label{P:unifholderineq}
Let $0<s_0< s_1\leq 1$ and let $s \in [s_0,s_1)$, $n>2s_1$. Let $0 < R_0\leq R$, $g \in L^\infty(B_R)$ and $v$ be a weak solution of 
\[
\begin{cases}
(-\Delta)^s v = g & \text{ in }B_R,\\
v = 0 & \text{ in }\R^n \setminus B_R,
\end{cases}
\]
Then $v \in C^{0,s}(\R^n)$ and 
\[
\|v\|_{C^{0,s}(\R^n)} \leq C|g|_{L^\infty(B_R)}
\]
where the constant $C>0$ depends only on $n$, $s_0$, $s_1$ and $R_0$, but neither on $s$ nor on $R$.
\end{proposition}
\begin{proof}
The estimate is a consequence of results contained in \cite{RosSer}. Concerning the dependence on the parameters $s_0$, $s_1$, it can be deduced from a careful analysis of the proof in \cite{RosSer} (see also \cite[Proposition 2.3]{CorIac}). As for the dependence of the constant $C$ on the domain, it turns out that $C$ depends only on the radii coming from the outer and inner ball conditions for $B_R$. Hence, it is clear that $C$ can be chosen in a uniform way with respect to $R$ if we assume that $R\geq R_0$, for some $R_0>0$. 
\end{proof}

From now on $u_{s, \varepsilon} \in \mathcal{M}^r_{s, \varepsilon}(B_R)$ will denote a least energy radial solution to Problem \eqref{fracradsubcrit}, i.e. $I_{s,\varepsilon}(u_{s,\varepsilon}) = C_{\mathcal{M}^r(B_R)}(s, \varepsilon)$. Moreover, we set $M_{s, \varepsilon} : = |u_{s, \varepsilon}|_\infty$. 
In the next result we characterize the asymptotic behavior of the points where the blow-up occurs. 

\begin{lemma}\label{L:maxpointbehavior}
Let $s \in (0,1)$, $n>2s$. Let $x_\varepsilon \in B_R$ be such that $|u_{s, \varepsilon}(x_\varepsilon)| = O(M_{s, \varepsilon})$ as $\varepsilon \to 0^+$. Then 
\[
M_{s, \varepsilon}^{\beta_{s, \varepsilon}} |x_\varepsilon|  \not \to + \infty, 
\]
where $\beta_{s, \varepsilon}:= \frac{2}{n+2s} - \frac{\varepsilon}{2s}$. In particular, we infer that $|x_\varepsilon| \to 0$.
\end{lemma}
\begin{proof}
If $s \in \left(\frac{1}{2},1\right)$ this is a consequence of the fractional Strauss inequality \eqref{straussineq} (see \cite[Proposition 1]{choozawa}). Indeed, suppose that $x_\varepsilon \neq 0$ (otherwise there is nothing to prove). Then
\[
\begin{aligned}
(M_{s, \varepsilon}^{\beta_{s, \varepsilon}} |x_\varepsilon|)^{\frac{n-2s}{2}} &= M_{s, \varepsilon}^{1-\frac{\varepsilon}{2^*_s-2}} |x_\varepsilon|^{\frac{n-2s}{2}}\\
 &\leq M_{s, \varepsilon}|x_\varepsilon|^{\frac{n-2s}{2}} \leq C|u_{s, \varepsilon}(x_\varepsilon)||x_\varepsilon|^{\frac{n-2s}{2}} \leq C K_{n,s}\|u_{s, \varepsilon}\|_s^2 \leq C, 
\end{aligned}
\]
where, in view of Lemma \ref{As:energas}-$(v)$, we used that $M_{s, \varepsilon} \geq 1$. Unfortunately, as pointed out in Sect. 1, the fractional Strauss inequality does not hold in general when $s \in \left(0,\frac{1}{2}\right]$. To overcome this difficulty we use the following argument, which is valid for any $s\in (0,1)$. 

Assume by contradiction that there exists a subsequence (still denoted by $\ep$ for simplicity), such that $M_{s, \varepsilon}^{\beta_{s, \varepsilon}} |x_\varepsilon|  \to + \infty$ as $\ep\to 0^+$.
Let us define the rescaled functions
\begin{equation}\label{E:rescaling}
\tilde u_{s, \varepsilon}(x) = \frac{1}{M_{s, \varepsilon}}u_{s, \varepsilon}\left( \frac{x}{M_{s, \varepsilon}^{\beta_{s, \varepsilon}}}\right), \quad x \in \R^n.
\end{equation}
It is immediate to see that the functions $\tilde u_{s, \varepsilon}$ satisfy
\begin{equation}\label{approxpass1}
\begin{cases}
(-\Delta)^s \tilde u_{s, \varepsilon} = |\tilde u_{s, \varepsilon}|^{2^*_s-2-\varepsilon}\tilde u_{s, \varepsilon}, & \text{ in }B_{M_{s, \varepsilon}^{\beta_{s, \varepsilon}} R},\\
\tilde u_{s, \varepsilon} = 0 & \text{ in }\R^n \setminus B_{M_{s, \varepsilon}^{\beta_{s, \varepsilon}} R}.
\end{cases}
\end{equation}
Since by construction $|\tilde u_{s, \varepsilon}|_\infty \leq 1$ and $M_{s, \varepsilon}^{\beta_{s, \varepsilon}} |x_\varepsilon|  \to + \infty$, then from Proposition \ref{P:unifholderineq} we deduce that
\begin{equation}\label{uppbound}
\|\tilde u_{s, \varepsilon}\|_{C^{0,s}(\R^n)} \leq C,
\end{equation}
for some $C>0$ independent on $\ep$. 

Now we observe that, by definition of $x_\ep$, there exists $C_1 \in (0, 1]$ such that for all sufficiently small $\ep>0$ it holds
\begin{equation}\label{uppbound2}
|\tilde u_{s, \varepsilon}(M_{s, \varepsilon}^{\beta_{s, \varepsilon}} x_\varepsilon)| = \frac{|u_{s, \varepsilon}(x_\varepsilon)|}{M_{s, \varepsilon}}\geq C_1 + o(\varepsilon).
\end{equation}
Using \eqref{uppbound}, \eqref{uppbound2} and the triangle inequality, for every $\tau >0$ and $\xi \in \R^n$ such that $|\xi| \leq 1$, we infer that
\[
\frac{C_1- |\tilde u_\lambda(M_{s, \varepsilon}^{\beta_{s, \varepsilon}} x_\varepsilon+ \tau \xi)| + o(\varepsilon)}{|\tau\xi|^s} \leq  \frac{|\tilde u_{s, \varepsilon}(M_{s, \varepsilon}^{\beta_{s, \varepsilon}} x_\varepsilon)- \tilde u_{s, \varepsilon}(M_{s, \varepsilon}^{\beta_{s, \varepsilon}} x_\varepsilon+ \tau \xi)|}{|\tau\xi|^s} \leq C.
\] 
In particular, we can find $\varepsilon_0>0$, $\tau_0>0$ and $C_2>0$ such that, for every $\varepsilon \in (0, \varepsilon_0)$, it holds 
\[
0 < C_2 \leq C_1-C\tau_0^s + o(\varepsilon)  \leq |\tilde u_{s, \varepsilon}(M_{s, \varepsilon}^{\beta_{s, \varepsilon}} x_\varepsilon+ \tau_0 \xi)|, \quad \forall |\xi|\leq 1.
\]
Therefore, since $\tilde u_{s, \varepsilon}$ is radial and $\xi$ is arbitrary we obtain that
\[
0 < C_2 \leq |\tilde u_{s, \varepsilon}(x)|,  \quad \forall x \in B_{M_{s, \varepsilon}^{\beta_{s, \varepsilon}} |x_\varepsilon| + \tau_0}\setminus B_{M_{s, \varepsilon}^{\beta_{s, \varepsilon}} |x_\varepsilon| - \tau_0}.
\]
Now, since we are assuming by contradiction that $M_{s, \varepsilon}^{\beta_{s, \varepsilon}} |x_\varepsilon| \to +\infty$ and since $M_{s, \varepsilon} \to +\infty$ we get that
\[
\begin{aligned}
|u_{s, \varepsilon}|^{2^*_s-\varepsilon}_{2^*_s-\varepsilon} &\geq M_{s, \varepsilon}^{-\varepsilon\left(\frac{n-2s}{2s} \right)}|u_{s, \varepsilon}|^{2^*_s-\varepsilon}_{2^*_s-\varepsilon} =  |\tilde u_{s, \varepsilon}|_{2^*_s-\varepsilon}^{2^*_s-\varepsilon}\\
 &\geq C_2[(M_{s, \varepsilon}^{\beta_{s, \varepsilon}}|x_\varepsilon| + \tau_0)^n-(M_{s, \varepsilon}^{\beta_{s, \varepsilon}}|x_\varepsilon| - \tau_0)^n] \\
&= 2nC_2 \tau_0 (M_{s, \varepsilon}^{\beta_{s, \varepsilon}} |x_\varepsilon|)^{n-1} + \psi(\varepsilon), 
\end{aligned}
\]
where $\psi(\varepsilon)$ is such that $\frac{\psi(\varepsilon)}{(M_{s, \varepsilon}^{\beta_{s, \varepsilon}}|x_\varepsilon|)^{n-1}}\to 0$, as $\ep\to 0^+$. From this we get that $| u_{s, \varepsilon}|_{2^*_s-\varepsilon}^{2^*_s-\varepsilon}\to +\infty$, as $\varepsilon \to 0^+$, which contradicts Lemma \ref{As:energas}, (ii). The proof is complete. 
\end{proof}

In the next result we study the asymptotic behavior of the rescaled solutions defined in \eqref{E:rescaling}.

\begin{lemma}\label{L:asyntrescal}
Let $s \in (0,1)$, $n>2s$ and let $(\tilde u_{s, \varepsilon})_\ep$ be the sequence of rescaled functions associated to $(u_{s,\ep})_\ep$, defined in \eqref{E:rescaling}. Then, up to a subsequence, $\tilde u_{s, \varepsilon} \to \tilde u_s$ in $C^{0,\alpha}_{loc}(\R^n)$ for some $\alpha \in (0,s)$, as $\ep\to 0^+$, where $\tilde u_s \in \D^s(\R^n)$ is a nontrivial weak solution to
\begin{equation}\label{eq:criticalproblemrn}
\begin{cases}
(-\Delta)^s \tilde u_s = |\tilde u_s|^{2^*_s-2}\tilde u_s & \text{ in }\R^n, \\
\tilde u_s >0.
\end{cases}
\end{equation}
Moreover, $\tilde u_s$ is radial and $|\tilde u_s|_\infty = |\tilde u_s(0)|$. 
\end{lemma}
\begin{proof}
As seen in the proof of Lemma \ref{L:maxpointbehavior}, the functions $\tilde u_{s, \varepsilon}$ weakly satisfy \eqref{approxpass1} and by construction it holds that $|\tilde u_{s, \varepsilon}|_\infty \leq 1$. Then, since $M_{s, \varepsilon}^{\beta_{s, \varepsilon}} R \to + \infty$, thanks to Proposition \ref{P:unifholderineq} and a standard argument, up to a subsequence, we have
\[
\tilde u_{s, \varepsilon} \to \tilde u_s \text{ in } C^{0,\alpha}_{loc}(\R^n),
\]
for some $\tilde u_s \in C^{0,\alpha}_{loc}(\R^n)$, $\alpha \in (0,s)$.
We point out that $\tilde u_s \not \equiv 0$. Indeed, let $x_\varepsilon \in B_R$ be such that $|u_{s, \varepsilon}(x_\varepsilon)| = M_{s, \varepsilon}$. By construction we have $|\tilde u_{s, \varepsilon}(M_{s, \varepsilon}^{\beta_{s, \varepsilon}} x_\varepsilon)|=1$, and thanks to Lemma \ref{L:maxpointbehavior} we infer that the point $M_{s, \varepsilon}^{\beta_{s, \varepsilon}} x_\varepsilon$ stays in a compact subset of $\R^n$. Therefore, from the $C^{0, \alpha}_{loc}$-convergence of $\tilde u_{s, \varepsilon}$ in $\R^n$, we get that $\tilde u_s$ is non trivial.

Now we show that $\tilde u_s \in \D^s(\R^n)$. In fact, by Lemma \ref{As:energas}-(i) and since $M_{s, \varepsilon} \to +\infty$, we infer that
\[
\|\tilde u_{s, \varepsilon}\|^2_s = M_{s, \varepsilon}^{-\varepsilon\left(\frac{n-2s}{2s}\right)}\|u_{s, \varepsilon}\|^2_s \leq \|u_{s, \varepsilon}\|^2_s \to 2S_s^{\frac{n}{2s}}, \ \ \text{as} \ \varepsilon \to 0^+,
\]
and in particular, up to a subsequence, $\tilde u_{s, \varepsilon} \rightharpoonup v$ for some $v \in \D^s(\R^n)$. Then, since $\tilde u_{s, \varepsilon} \to \tilde u_s$ in $C^{0, \alpha}_{loc}(\R^n)$, we get that $v=\tilde u_s$ and we are done. In addition, applying Fatou's Lemma we also deduce that
\begin{equation}\label{nosc}
\|\tilde u_s\|_s^2 \leq \liminf_{\ep \to 0^+}\|\tilde u_{s, \varepsilon}\|^2_s \leq 2S_s^{\frac{n}{2s}}.
\end{equation}

Let us prove now that $\tilde u_s$ is a weak solution to \eqref{eq:criticalproblemrn}. Indeed, for every $\varphi \in C^\infty_c(\R^n)$, since $\tilde u_{s,\ep}$ is a weak solution to \eqref{approxpass1} we have
\begin{equation}\label{eq:weaklimriscal}
(\tilde u_{s, \varepsilon}, \varphi)_s =  \int_{B_{M_{s, \varepsilon}^{\beta_{s, \varepsilon}}R}}|\tilde u_{s, \varepsilon}|^{2^*_s-2-\varepsilon}\tilde u_{s, \varepsilon} \varphi \de x,
\end{equation}
where $\varepsilon$ is small enough  so that $\text{supp }\varphi \subset B_{M_{s, \varepsilon}^{\beta_{s, \varepsilon}} R}$.
Since $\tilde u_{s, \varepsilon} \to \tilde u_s$ for a.e. $x \in \R^n$, using the well known relations (see e.g. \cite{CorIac})
\begin{equation}\label{E:fracintbypart}
(u, \varphi)_s = \int_{\R^n}u (-\Delta)^s \varphi \de x, \quad \forall u \in \D^s(\R^n), \ \forall \varphi \in C^\infty_c(\R^n),
\end{equation}
and
\[
|(-\Delta)^s \varphi (x)| \leq C(\varphi)\frac{1}{(1+|x|)^{n+2s}},\quad \forall x \in \R^n,\ \forall \varphi \in C^\infty_c(\R^n), 
\]
and thanks to Lebesgue's dominated convergence theorem, passing to the limit as $\ep\to 0^+$ in \eqref{eq:weaklimriscal} we infer that
\[
\int_{\R^n}\tilde u_s (-\Delta)^s \varphi \de x = \int_{\R^n}|\tilde u_s|^{2^*_s-2}\tilde u_s \varphi \de x.
\]
Now, since $\tilde u_s \in \D^s(\R^n)$ we are allowed to use again \eqref{E:fracintbypart}, obtaining that $\tilde u_s$ weakly satisfies
\begin{equation}\label{weaksobsol}
(-\Delta)^s \tilde u_s = |\tilde u_s|^{2^*_s-2}\tilde u_s \quad \text{ in }\R^n.
\end{equation}
We prove now that $\tilde u_s$ is of constant sign. To this end, assume by contradiction that $\tilde u_s$ is sign-changing. Then, using $\tilde u_s^\pm \in \D^s(\R^n)$ as test functions in \eqref{weaksobsol} and recalling that $(\tilde u_s^+, \tilde u_s^-)_s < 0$, we get that
\[
\|\tilde u_s^\pm\|^2_s = (\tilde u_s^+, \tilde u_s^-)_s + |\tilde u_s^\pm|^{2^*_s}_{2^*_s} <|\tilde u_s^\pm|^{2^*_s}_{2^*_s}.
\]
Hence, by the Sobolev inequality we infer that
\[
S_s \leq \frac{\|\tilde u_s^\pm\|^2_s}{|\tilde u_s^\pm|^{2}_{2^*_s}}< |\tilde u_s^\pm|^{2^*_s -2}_{2^*_s},  
\]
and thus $2S_s^{\frac{n}{2s}} < |\tilde u_s^+|^{2^*_s}_{2^*_s}+ |\tilde u_s^-|^{2^*_s}_{2^*_s} = |\tilde u_s|^{2^*_s}_{2^*_s}$. 
Finally, using $\tilde u_s$ as a test function in \eqref{weaksobsol} we have $\|\tilde u_s\|_s^2 = |\tilde u_s|_{2^*_s}^{2^*_s}$, and we obtain that $2S_s^{\frac{n}{2s}}< \|\tilde u_s\|_s^2$, which contradicts \eqref{nosc}.

At the end we notice that, since $\tilde u_s$ is a pointwise limit of radial functions, it is radial too. Moreover, since $\tilde u_s$ is of constant sign, assuming without loss of generality that $\tilde u_s \geq 0$, we easily deduce, by the fractional strong maximum principle and the fractional moving plane method (see \cite{ChenLiLi}), that $\tilde u_s$ is also decreasing along the radii and thus $\tilde u_s$ achieves its maximum at the origin. The proof is complete. 
\end{proof}

A consequence of the above result is that least energy radial sign-changing solutions to \eqref{fracradsubcrit} cannot vanish at the origin.

\begin{lemma}\label{L:nozeroorigin}
Let $s\in (0,1)$ and $n>2s$. There exist $\overline \varepsilon >0$ and $C > 0 $ such that for every $\varepsilon \in (0, \overline  \varepsilon)$ it holds
\begin{equation}\label{eq:estimateatzero}
 |u_{s, \varepsilon}(0)|\geq C.
 \end{equation}
Moreover, for every $0<s_0<s_1 \leq 1$, $n>2s_1$ estimate \eqref{eq:estimateatzero} holds with $\overline \ep>0$, $C>0$ independent on $s \in [s_0,s_1)$. 
\end{lemma}
\begin{proof}
We prove directly the second part of the Lemma.  Assume by contradiction that there exist $0<s_0<s_1 \leq 1$, three sequences $\varepsilon_k \to 0^+$, $C_k \to 0^+$, $(s_k)_k \in [s_0, s_1)$, and a sequence of nodal radial least energy solutions $u_k:= u_{s_k, \varepsilon_k}$  such that $|u_k(0)| <C_k$. Up to a subsequence, we can always assume that $s_k \to \sigma$, with $\sigma \in [s_0, s_1]$. 

Now, only two possibilities can occur: setting $M_k := |u_k|_\infty$, either $(M_k)_k$ is a bounded sequence or there exists a subsequence such that $M_k \to + \infty$. 

Assume that $(M_k)_k$ is bounded. We first observe that $(M_k)_k$ is bounded away from zero, otherwise we could find a subsequence such that $M_k \to 0$, but this would contradict Lemma \ref{As:energas}. Therefore, up to a subsequence we can assume that $M_k \to l$, for some real number $l>0$. Adapting the arguments of Lemma \ref{L:asyntrescal} and using Lemma \ref{L:nodalenerasynt} we readily infer that, up to a subsequence, $u_k \rightharpoonup u$ in $X^{s_0}_0(B_R)$ and $u_k \to u$ in $C^{0, \alpha}(\R^n)$, for some $\alpha \in (0,s_0)$. Furthermore we have $u \not \equiv 0$ and it holds that
\[
\int_{\R^n}u (-\Delta)^\sigma \varphi \de x= \int_{\R^n}|u|^{2^*_\sigma -2}u\varphi \de x \quad \forall \varphi \in C^\infty_c(B_R). 
\]
Using that $u_k \to u$ in $L^2(B_R)$, thanks to the fractional Sobolev embedding and Fatou's Lemma we find
\begin{equation}\label{fatoufourier}
\|u\|^2_\sigma = \int_{\R^n}|\xi|^{2\sigma}|\hat u(\xi)|^2\de x \leq \liminf_{k \to + \infty}\int_{\R^n}|\xi|^{2s_k}|\hat u_k(\xi)|^2\de x = \liminf_{k \to + \infty} \|u\|_s^2 \leq \frac{2\sigma}{n}S_\sigma^{\frac{n}{2\sigma}},
\end{equation}
where the last inequality is a consequence of the second part of Lemma \ref{L:nodalenerasynt}, while the equalities are due to the interpretation via the Fourier transform of the fractional Laplacian (see e.g. \cite{Hitch}).
From this discussion it follows that $u$ is a non trivial weak solution of 
\begin{equation}\label{critprob}
\begin{cases}
(-\Delta)^\sigma u = |u|^{2^*_\sigma -2}u & \text{in }B_R, \\
u = 0 & 	\text{in }\R^n \setminus B_R.
\end{cases}
\end{equation}
This readily contradicts the Pohozaev identity when $\sigma = 1$. If $\sigma <1$, the fractional Pohozaev identity only implies the nonexistence of constant-sign solutions to \eqref{critprob} (see \cite{RosSerrI}).  In order to obtain a contradiction we show that $u$ is of constant sign. Indeed, arguing as in the proof of Lemma \ref{L:asyntrescal}, we have that any sign-changing solution $u$ to \eqref{critprob} must satisfy $\|u\|^2_\sigma >\frac{2\sigma}{n}S_\sigma^{\frac{n}{2\sigma}}$. Hence, thanks to \eqref{fatoufourier} it follows that $u$ is of constant sign and we get the desired contradiction.
\medskip

Let us analyze the second case. Assume that $M_k \to + \infty$ and consider the rescaled functions
\[
\tilde u_k(x) = \frac{1}{M_k}u_k\left( \frac{x}{M_{k}^{\beta_{k}}}\right), \quad x \in \R^n,
\]
where $\beta_k = \frac{2}{n-2s_k} - \frac{\varepsilon_k}{2s_k}$. 
Arguing as in Lemma \ref{L:asyntrescal}, and taking into account Lemma \ref{L:nodalenerasynt}, we obtain that $\tilde u_{k} \to \tilde u$ in $C^{0,\alpha}_{loc}(\R^n)$, as $k \to +\infty$, for some $\alpha \in (0,s_0)$, where the function $\tilde u$ belongs to $\D^\sigma(\R^n)\cap C^{0,\alpha}_{loc}(\R^n) \setminus\{0\}$, it is radial, verifies $|\tilde u|_\infty = |\tilde u(0)|$ and weakly satisfies 
\[
\begin{cases}
(-\Delta)^\sigma \tilde u = |\tilde u|^{2^*_\sigma-2}\tilde u & \text{ in }\R^n, \\
\tilde u >0.
\end{cases}
\]
The only delicate point is when $\sigma =1$. Indeed, in this case we cannot simply argue via Fatou's Lemma to show that $\tilde u \in \D^1(\R^n)$. Nevertheless, since 
\[
|\tilde u|_{2^*_1}^{2^*_1} \leq \liminf_{k \to + \infty} |\tilde u_k|^{2^*_{s_k}}_{2^*_{s_k}}\leq \liminf_{k \to + \infty}S_{s_k}^{-\frac{2^*_{s_k}}{2}}\|\tilde u_k\|_{s_k}^{\frac{2^*_{s_k}}{2}}\leq C, 
\]
we have that $\tilde u \in L^{2^*_1}(\R^n)$. Therefore we can apply \cite[Theorem 2, Corollary 3]{farina}, obtaining that $\tilde u \in \D^1(\R^n)$ and $\|\tilde u\|^2_1 = |\tilde u|^{2^*_1}_{2^*_1}$.

To conclude notice that, since we are assuming $|\tilde u_k(0)|= \frac{1}{M_k}|u_k(0)|\to 0$ and since we have $\tilde u_k \to \tilde u$ in $C^{0, \alpha}_{loc}(\R^n)$, then it follows that $|\tilde u(0)|=0$, which contradicts the non triviality of $\tilde u$. The proof is then complete.
\end{proof}

In the next lemma we show, independently on the number of sign-changes, that $M_{s, \varepsilon}$ is achieved in the  nodal component containing the origin and  blows up faster than every other extremal value achieved in the other components. Before stating the result we introduce some notation. Assuming without loss of generality that $u_{s, \varepsilon}(0) > 0$, thanks to Lemma \ref{L:nozeroorigin}, for all sufficiently small $\ep>0$ the following quantities are well defined:
\[
\begin{aligned}
&r^1_\varepsilon: = \min\{r \in (0,R] \ ; \ u_{s, \varepsilon}(x) = 0, |x|= r\},\\
&M^+_{s, \varepsilon}: = \max \{u_{s, \varepsilon}(x) \ ; \ 0 \leq |x| \leq r^1_\varepsilon\}, \\
&\hat M_{s, \varepsilon}: = \max \{|u_{s, \varepsilon}(x)| \ ; \ r^1_\varepsilon \leq |x| \leq R\}.
\end{aligned}
\] 
In other words, $r^1_\varepsilon$ is the first nodal radius, $M_{s, \varepsilon}^+$ is the maximum of the solution in the first nodal component, while $\hat M_{s, \varepsilon}$ is the absolute maximum achieved in the other nodal components. 

\begin{lemma}\label{L:maxasynt}
Let $s\in (0,1)$ and $n>2s$. There exists $\varepsilon'>0$ such that for every $\varepsilon \in (0, \varepsilon')$ it holds
\[
M_{s, \varepsilon} = M_{s, \varepsilon}^+.
\]
Moreover, up to a subsequence, as $\varepsilon \to 0^+$
\[
(i) \ \ (M_{s, \varepsilon}^+)^{\beta_{s, \varepsilon}} r_\varepsilon^1 \to +\infty,\quad (ii) \ \ \frac{M_{s, \varepsilon}^+}{\hat M_{s, \varepsilon}}\to + \infty.  
\]
\end{lemma}
\begin{proof}
We begin by proving that $M_{s, \varepsilon} = M_{s, \varepsilon}^+$. 
Suppose by contradiction that there exists a sequence $\varepsilon \to 0^+$ such that $M_{s, \varepsilon} = \hat M_{s, \varepsilon}$. By Lemma \ref{L:maxpointbehavior} we get that $ M_{s, \varepsilon}^{\beta_{s, \varepsilon}} |x_\varepsilon| \not \to \infty$, where $x_\varepsilon$ is any point such that $|u_{s, \varepsilon}(x_\varepsilon) | = M_{s, \varepsilon}$. Notice that by construction we have $|x_\varepsilon| \geq r^1_\varepsilon$ and thus $M_{s, \varepsilon}^{\beta_{s, \varepsilon}} r^1_\varepsilon \not \to + \infty$ too.

As a consequence, up to a subsequence, $M_{s, \varepsilon}^{\beta_{s, \varepsilon}} r^1_\varepsilon\to l$ for some real number $l\geq 0$. Let $\tilde u_{s, \varepsilon}$ be the rescaling defined in \eqref{E:rescaling}. Then, by Lemma \ref{L:asyntrescal} we infer that $\tilde u_{s, \varepsilon}\to \tilde u_s$ in $C^{0,\alpha}_{loc}(\R^n)$ for some $\alpha \in (0,s)$, $ \tilde u_s \in C^{0,\alpha}_{loc}(\R^n)$. On the other hand, let $(y_\varepsilon) \subset \R^n$ be such that $|y_\varepsilon| = r^1_\varepsilon$. Up to a further subsequence, $M_{s, \varepsilon}^{\beta_{s, \varepsilon}} y_\varepsilon \to \hat y$ as $\ep \to 0^+$, where $|\hat y | = l$. 
Then, thanks to Proposition \ref{P:unifholderineq} and since $\tilde u_{s, \varepsilon}(M_{s, \varepsilon}^{\beta_{s, \varepsilon}} y_\varepsilon)=0$, $\tilde u_{s, \varepsilon} \to \tilde u_s$ a.e., we get that
\[
|\tilde u_s (\hat y)| \leq |\tilde u_{s, \varepsilon}(\hat y) - \tilde u_{s, \varepsilon}(M_{s, \varepsilon}^{\beta_{s, \varepsilon}} y_\varepsilon)| + |\tilde u_{s, \varepsilon}(\hat y) - \tilde u_s(\hat y)| \leq  C|\hat y - M_{s, \varepsilon}^{\beta_{s, \varepsilon}} y_\varepsilon|^\alpha + o(1) = o(1).
\] 
From this we deduce that $\tilde u(\hat y) = 0$, which contradicts the strict positivity of $\tilde u$ (see Lemma \ref{L:asyntrescal}). 

The proof of (i)  is identical and we omit it. Let us prove (ii). Let $x_\varepsilon \in \R^n$ be such that $|u_{s, \varepsilon}(x_\varepsilon)|= \hat M_{s, \varepsilon}$. We claim that  $\tilde u_{s, \varepsilon}(M_{s, \varepsilon}^{\beta_{s, \varepsilon}} x_\varepsilon)\to 0$.

 Indeed, if it is not the case, up to a subsequence, we find $c \in (0,1]$ such that  $\tilde u_{s, \varepsilon}(M_{s, \varepsilon}^{\beta_{s, \varepsilon}} x_\varepsilon)\to c$. Thanks to Proposition \ref{P:unifholderineq} and arguing  as in the proof of Lemma \ref{L:maxpointbehavior} we obtain that there exists a positive constant $C_1$ and a small positive number $\tau_0$, both independent on $\ep$, such that for all sufficiently small $\ep>0$
\[
|\tilde u_{s, \varepsilon}| \geq C_1 >0,\ \  \forall x \in B_{M_{s, \varepsilon}^{\beta_{s, \varepsilon}} |x_\varepsilon|+ \tau_0}\setminus B_{M_{s, \varepsilon}^{\beta_{s, \varepsilon}} |x_\varepsilon|-\tau_0}.
\]
Since $M_{s, \varepsilon}^{\beta_{s, \varepsilon}}|x_\varepsilon|>M_{s, \varepsilon}^{\beta_{s, \varepsilon}} r_\varepsilon^1 \to +\infty$, because of (i), we obtain that $|\tilde u_{s, \varepsilon}|^{2^*_s-\varepsilon}_{2^*_s-\varepsilon} \to + \infty$, which contradicts Lemma \ref{As:energas}-(ii). The claim is thus proved.

Now, in order to conclude the proof of (ii), we notice that
\[
1 = \frac{ |u_{s, \varepsilon}(x_\varepsilon)|}{\hat M_{s, \varepsilon}} = \frac{|\tilde u_{s, \varepsilon}(M_{s, \varepsilon}^{\beta_{s, \varepsilon}}x_\varepsilon)|M_{s, \varepsilon}}{\hat M_{s, \varepsilon}},
\]
and thanks to the previous claim we obtain the desired result.

\end{proof}

An immediate consequence of the previous result is the following
\begin{lemma}\label{positivebubble}
Let $s\in (0,1)$ and $n>2s$. The function $\tilde u_s$ given by Lemma \ref{L:asyntrescal} is a standard bubble centered at the origin, i.e. $\tilde u_s$ is of the form \eqref{fracbubble} with $x_0=0$.
\end{lemma}
\begin{proof}
Since $\tilde u_s$ is radial and satisfies \eqref{eq:criticalproblemrn}, in order to prove the desired result we only need to show that $\tilde u_s$ realizes the infimum in the fractional Sobolev inequality, that is, 
\[
S_s = \frac{\|\tilde u_s\|^2_s}{|\tilde u_s|^2_{2^*_s}}.
\]
Now, since $\|\tilde u_s\|^2_s = |\tilde u_s|^{2^*_s}_{2^*_s}$, by the Sobolev inequality, we readily infer that
\[
S_s \leq \frac{\|\tilde u_s\|^2_s}{|\tilde u_s|^{2}_{2^*_s}}\leq |\tilde u_s|^{2^*_s-2}_{2^*_s}, 
\]
and to conclude it suffices to show that $|\tilde u_s|^{2^*_s}_{2^*_s}\leq S_s^{\frac{n}{2s}}$.
To this end, we set
\[
u_{s, \varepsilon}^1(x) := 
\begin{cases}
u_{s, \varepsilon}(x) & |x|\leq r^1_\varepsilon, \\
0 & \text{otherwise},
\end{cases}
\]
and define
\[
\tilde u_\varepsilon^1 (x):= \frac{1}{M_{s, \varepsilon}^+}u_{s, \varepsilon}^1\left(\frac{x}{(M_{s, \varepsilon}^+)^{\beta_{s, \varepsilon}}}\right), \quad x \in \R^n.
\]
Thanks to Lemma \ref{L:maxasynt}, for all sufficiently small $\ep>0$ we have $M_{s, \varepsilon} = M_{s, \varepsilon}^+$. Then we readily infer that $|\tilde u_{s, \varepsilon} - \tilde u_{s, \varepsilon}^1|_\infty \leq \frac{\hat M_{s, \varepsilon}}{M_{s, \varepsilon}}$ which, again by Lemma \ref{L:maxasynt}, implies that $\tilde u_{s, \varepsilon}- \tilde u_{s, \varepsilon}^1\to 0$ uniformly in $\R^n$, as $\varepsilon \to 0^+$. Then, by Fatou's Lemma, taking into account that $|\tilde u^1_{s, \varepsilon}|^{2^*_s-\varepsilon}_{2^*_s-\varepsilon}=\left(M_{s,\ep}^+\right)^{-\ep \left(\frac{n-2s}{2s}\right)} | u^1_{s, \varepsilon}|^{2^*_s-\varepsilon}_{2^*_s-\varepsilon}$, $M_{s,\ep}^+\geq 1$ and Lemma \ref{As:energas}, we get that
\[
|\tilde u_s|^{2^*_s}_{2^*_s} \leq \liminf_{\ep \to 0^+}|\tilde u^1_{s, \varepsilon}|^{2^*_s-\varepsilon}_{2^*_s-\varepsilon} \leq \liminf_{\ep \to 0^+}| u^1_{s, \varepsilon}|^{2^*_s-\varepsilon}_{2^*_s-\varepsilon} \leq \liminf_{\ep \to 0^+}|u^+_{s, \varepsilon}|^{2^*_s-\varepsilon}_{2^*_s-\varepsilon} = S_s^{\frac{n}{2s}}.
\]
The proof is then complete.

\end{proof}
\end{section}

\begin{section}{Characterization of the nodal set}
In this section we study the nodal set of least energy radial sign-changing solutions to Problem \eqref{fracradsubcrit}. We begin with a couple of known preliminary results, which provide, respectively, an upper bound on the number of sign changes and a characterization of the nodal set.
\begin{lemma}\label{changebound}
Let $n>2s$, $s \in (0,1)$. Let $u_{s,\varepsilon}$ be a least energy radial sign-changing solution to Problem \eqref{fracradsubcrit}. There exists $\tilde \varepsilon_s \in (0, 2^*_s-2)$ such that, if $\varepsilon \in (0, \tilde \varepsilon _s)$, then $u_{s, \varepsilon} = u_{s,\varepsilon}(r)$ changes sign at most twice.

Let $0<s_0< s_1\leq 1$ and $n>2s_1$. Then there exists $\tilde \varepsilon>0$, independent on $s$, such that the same result holds for every $s \in [s_0, s_1)$ and $\varepsilon \in (0, \tilde \varepsilon)$.
\end{lemma}
\begin{proof}
It suffices to argue as in \cite[Theorem 5.1]{CorIac} first, and then as in \cite[Theorem 5.2]{CorIac}, taking into account Lemma \ref{L:nodalenerasynt} and Lemma \ref{As:energas}. In particular $\hat \varepsilon>0$ is  given by Lemma \ref{L:nodalenerasynt}.
\end{proof}


\begin{lemma}\label{L:nodalchar}
Let $s \in \left(0, 1\right)$ and $n>2s$. There exists $\check \varepsilon_s>0$ such that  for all $\varepsilon \in (0, \check \varepsilon_s)$ any least energy radial sign-changing solution $u_{s,\varepsilon}$ to \eqref{fracradsubcrit} vanishes only at the nodes.

Moreover, let $0<s_0< s_1 \leq 1$, $n >2s_1$. Then the above result hold true for every $s \in [s_0, s_1)$ and $\varepsilon \in (0, \check \varepsilon)$, for some $\check \varepsilon>0$ independent on $s$.
\end{lemma}
\begin{proof}
It suffices to take ${\check \varepsilon}_s := \min\{\overline \varepsilon_s, \tilde \varepsilon_s\}$,
where $\overline \varepsilon_s, \tilde \varepsilon_s$ are given by Lemma \ref{L:nozeroorigin} and Lemma \ref{changebound}, respectively. Then, the results follows immediately by adapting the arguments of \cite[Theorem 1.2]{CorIac}. 
\end{proof}


In the next Lemma we prove the upper semi-continuity of the map $s \to C_{\mathcal{M}^r(B_R)}(s, \varepsilon)$.
\begin{lemma}\label{uppersemiconts}
Let $0<s_0<s_1 \leq 1$, $n >2s_1$ and $\varepsilon \in (0, \check \varepsilon)$, where $\check \varepsilon$ is given by in Lemma \ref{L:nodalchar}. Then for every $\sigma \in [s_0, s_1]$ we have
\[
\limsup_{s \to \sigma}C_{\mathcal{M}^r(B_R)}(s, \varepsilon) \leq C_{\mathcal{M}^r(B_R)}(\sigma, \varepsilon).
\]
\end{lemma}
\begin{proof}
Let us fix $s_0$, $s_1$, $n$ and $\ep$ as in the statement. Let $(s_k)_k \subset [s_0, s_1)$ be a sequence such that $s_k \to \sigma \in [s_0, s_1]$, and consider a radial solution $u_{\sigma, \varepsilon}$ of \eqref{fracradsubcrit} which realizes $C_{\mathcal{M}^r(B_R)}(\sigma, \varepsilon)$. Assume that $\sigma <1$.  We aim to construct a sequence of almost minimizers of $C_{\mathcal{M}^r}(s_k, \varepsilon)$. We proceed in three different steps. We point out that when $\sigma=1$ the proof is identical, taking into account the conventions $(-\Delta)^1 u = -\Delta u$, $\|u\|^2_1 = |\nabla u|^2_2$, and that $(u^+,u^-)_1 \equiv 0 $ for all $u \in H^1_0(B_R)$.
\medskip

\vspace{2mm}
\begin{cla}
There exists a sequence $(\varphi_j)_j \subset C^\infty_c(B_R) \cap \mathcal{M}^r_{\sigma,\varepsilon}(B_R)$ such that 
\begin{enumerate}
\item $\text{supp}(\varphi_j^\pm) \subset \text{supp }(u^\pm_{\sigma, \varepsilon})$,
\item $\varphi_j \to u_{\sigma, \varepsilon}$ in $X^\sigma_0(B_R)$, as $j \to +\infty$, 
\item $I_{\sigma, \varepsilon}(\varphi_j) \to I_{\sigma,\varepsilon}(u_{\sigma, \varepsilon})$, as $j \to +\infty$. 
\end{enumerate}
\end{cla}
\vspace{2mm}

We first observe that, thanks to Lemma \ref{L:nodalchar}, the boundaries of $\text{supp }(u^\pm_{\sigma, \varepsilon})$ consist in a finite union of spheres. Therefore, adapting known density results (see e.g. \cite{FisSerVal}) we find two sequences of radial functions $(\tilde \varphi_j^\pm)_j \subset C^\infty_c(B_R)$ such that $\tilde \varphi_j^\pm \geq 0$, $\text{supp}(\tilde \varphi_j^\pm) \subset \text{supp }(u^\pm_{\sigma, \varepsilon})$ for all $j$, and $\tilde \varphi^\pm_j \to u_{\sigma, \varepsilon}^\pm$ in $X^\sigma_0(B_R)$.
Observe that, from the continuity of the scalar product, we have $(\tilde \varphi_j^+, \tilde \varphi_j^-)_\sigma \to (u_{\sigma, \varepsilon}^+, u_{\sigma, \varepsilon}^-)_\sigma$. 

Now we recall that it is always possible to find $\alpha_j>0$, $\beta_j>0$ such that  $\alpha_j \tilde \varphi_j^+ - \beta_j \tilde \varphi_j^- \in \mathcal{M}^r_{\sigma, \varepsilon}(B_R)$ (see e.g. \cite[Remark 3.4]{CorIac}), which is equivalent to solving the following
\begin{equation}\label{neharicondition}
\begin{aligned}
\alpha_j^{2^*_\sigma-2-\varepsilon}|\tilde \varphi_j^+|^{2^*_\sigma-\varepsilon}_{2^*_\sigma-\varepsilon} +\frac{\beta_j}{\alpha_j}(\tilde \varphi_j^+, \tilde \varphi_j^-)_\sigma &= \|\tilde \varphi^+_j\|^2_\sigma,\\
\beta_j^{2^*_\sigma-2-\varepsilon}|\tilde \varphi_j^-|^{2^*_\sigma-\varepsilon}_{2^*_\sigma-\varepsilon} +\frac{\alpha_j}{\beta_j}(\tilde \varphi_j^+, \tilde \varphi_j^-)_\sigma &= \|\tilde \varphi^-_j\|^2_\sigma.
\end{aligned}
\end{equation}
We claim that, definitely, $0<\underline\alpha<\alpha_j<\overline\alpha$ and $0<\underline\beta<\beta_j<\overline\beta$, for some positive constants $\underline\alpha, \overline\alpha, \underline\beta, \overline\beta$.
 Indeed, since $\tilde \varphi^\pm_j \to u_{\sigma, \varepsilon}^\pm$, and $u_{\sigma, \varepsilon}^\pm$ are non trivial, then the quantities $|\tilde \varphi_j^\pm|^{2^*_\sigma-\varepsilon}_{2^*_\sigma-\varepsilon}$,  $\|\tilde \varphi^\pm_j\|^2_\sigma$, $(\tilde \varphi_j^+, \tilde \varphi_j^-)_\sigma$ are uniformly bounded and uniformly away from zero. Moreover, by the definition of the scalar product we always have $(\tilde \varphi_j^+, \tilde \varphi_j^-)_\sigma<0$. Then, treating \eqref{neharicondition} as an algebraic system in $\alpha_j$, $\beta_j$ having as coefficients $|\tilde \varphi_j^\pm|^{2^*_\sigma-\varepsilon}_{2^*_\sigma-\varepsilon}$,  $\|\tilde \varphi^\pm_j\|^2_\sigma$, $(\tilde \varphi_j^+, \tilde \varphi_j^-)_\sigma$, it is easy to verify that, up to a sequence, it cannot happen that $\alpha_j \to +\infty$ or $\alpha_j\to 0^+$, and the same holds for $\beta_j$. The claim is thus proved.


Let us consider the sequence defined by $\varphi_j := \alpha_j \tilde \varphi_j^+ - \beta_j \tilde \varphi_j^-$.
By construction $(\varphi_j)_j \subset C^\infty_c(B_R)$, and in view of \eqref{neharicondition} we have $(\varphi_j)_j \subset \mathcal{M}^r_{\sigma, \varepsilon}(B_R)$. We claim that $\varphi_j \to u_{\sigma, \varepsilon}$ in $X^\sigma_0(B_R)$. 

Indeed, observe that, since $u_{\sigma, \varepsilon} \in \mathcal{M}_{\sigma, \varepsilon}^r(B_R)$, then
\[
\|u_{\sigma, \varepsilon}^\pm\|^2_\sigma = |u_{\sigma, \varepsilon}^\pm|^{2^*_\sigma-\varepsilon}_{2^*_\sigma-\varepsilon} + (u_{\sigma, \varepsilon}^+, u_{\sigma, \varepsilon}^-)_\sigma,
\]
then, up to a sequence, setting $\alpha:=\lim_{j \to +\infty} \alpha_j$, $\beta:=\lim_{j\to +\infty} \beta_j$ and passing to the limit in \eqref{neharicondition} we infer that
\[
\begin{aligned}
&\alpha(\alpha^{2^*_\sigma-2- \varepsilon}-1)|u_{\sigma, \varepsilon}^+|^{2^*_\sigma-\varepsilon}_{2^*_\sigma-\varepsilon} = -(\beta - \alpha) (u_{\sigma, \varepsilon}^+, u_{\sigma, \varepsilon}^-)_\sigma,\\
&\beta(\beta^{2^*_\sigma-2-\varepsilon}-1)|u_{\sigma, \varepsilon}^-|^{2^*_\sigma-\varepsilon}_{2^*_\sigma-\varepsilon} = - \alpha(\alpha^{2^*_\sigma-2-\varepsilon}-1)|u_{\sigma, \varepsilon}^+|^{2^*_\sigma-\varepsilon}_{2^*_\sigma-\varepsilon}.
\end{aligned}
\]
Recalling that $(u_{\sigma, \varepsilon}^+, u_{\sigma, \varepsilon}^-)_\sigma < 0$ it is immediate to see that both $0<\alpha<1$ and  $\alpha >1$ lead to a contradiction. Hence $\alpha =1$, and as a consequence we obtain that $\beta = 1$. Finally, from this and since
\[
\|u_{\sigma,\varepsilon}-\varphi_j\|_\sigma \leq |\alpha_j - 1|\|\tilde \varphi_j^+\|_\sigma + |\beta_j -1|\|\tilde \varphi_j^-\|_\sigma + \|\tilde \varphi^+_j - u_{\sigma, \varepsilon}^+\|_\sigma + \|\tilde \varphi_j^- - u_{\sigma, \varepsilon}^-\|_\sigma, 
\]
we obtain that $\varphi_j \to u_{\sigma, \ep}$ in $X^\sigma_0(B_R)$, as $j \to +\infty$. At the end, the last point of Step 1 is a straightforward consequence of the strong convergence of $\varphi_j$ to $u_{\sigma, \ep}$, together with the fractional Sobolev embedding. The proof of Step 1 is complete. 
\vspace{2mm}
\begin{cla}
Let $(\varphi_j)_j \in C^\infty_c(B_R) \cap \mathcal{M}^r_{\sigma,\ep}(B_R)$ be the sequence given by Step 1. Let $(s_k)_k$ be a sequence such that $s_k \to \sigma$, as $k \to +\infty$. 
For every $j$ fixed, we claim that there exists a sequence $(\varphi_{j, k})_k \subset C^\infty_c(B_R)$ such that $\varphi_{j,k} \in \mathcal{M}_{s_k, \varepsilon}^r(B_R)$ for every $k$, and
\[
\|\varphi_{j,k}\|_{s_k} \to \|\varphi_j\|_\sigma, \quad I_{s_k, \varepsilon}(\varphi_{j,k}) \to I_{\sigma,\varepsilon}(\varphi_j), \ \  \hbox{as} \ k \to + \infty . 
\]
\end{cla}
\vspace{4mm}
Let us fix $j$ and let $\varphi_j$ be as in the statement. From \eqref{E:fracintbypart} and \cite[Lemma 2.4]{uniquenondeg}, as $k \to +\infty$ we have $\|\varphi_j^\pm\|_{s_k} \to \|\varphi_j^\pm\|_\sigma$ and $\|\varphi_j\|_{s_k} \to \|\varphi_j\|_\sigma$. This easily implies that $(\varphi_j^+, \varphi_j^-)_{s_k} \to (\varphi_j^+, \varphi_j^-)_\sigma$, while by a standard computation we get that $|\varphi_j^\pm|_{2^*_{s_k}-\varepsilon}^{2^*_{s_k}-\varepsilon} \to |\varphi_j^\pm|^{2^*_\sigma-\varepsilon}_{2^*_\sigma-\varepsilon}$. 

Let $\alpha_k=\alpha(j,k)>0$, $\beta_k=\beta(j,k)>0$ be such that $\alpha_k \varphi_j^+ - \beta_k \varphi_j^- \in \mathcal{M}^r_{s_k, \ep}(B_R)$ and define $\varphi_{j,k} := \alpha_k \varphi_j^+ - \beta_k \varphi_j^-$. Arguing as in Step $1$ we get that, up to a subsequence, $\alpha_k, \beta_k \to 1$ as $k \to + \infty$. 
This easily implies that $\|\varphi_{j,k}\|_{s_k} \to \|\varphi_j\|_\sigma$ and $|\varphi_{j,k}|^{2^*_{s_k}-\varepsilon}_{2^*_{s_k}-\varepsilon} \to |\varphi_j|^{2^*_\sigma-\varepsilon}_{2^*_\sigma-\varepsilon}$, as $k \to +\infty$. The proof of step 2 is complete.
\vspace{2mm}
\begin{cla}
Conclusion.
\end{cla}
\vspace{2mm}
Let $(s_k)_k \subset (0,1)$ be a sequence such that $s_k \to \sigma$. Let us fix a small number $\tau >0$. Thanks to Step 1, there exists a function $\varphi_{\tau} \in C^\infty_c(B_R) \cap \mathcal{M}^r_{\sigma, \varepsilon}(B_R)$ such that 
\[
|I_{\sigma, \varepsilon}(u_{\sigma, \varepsilon}) - I_{\sigma, \varepsilon}(\varphi_{\tau})| < \frac{\tau}{2}.
\]
On the other hand, thanks to Step 2 there exist $\hat k= \hat k(\tau)>0$ and a sequence of functions $(\varphi_k)_k$ such that $\varphi_{k} \in C^\infty_c(B_R) \cap \mathcal{M}^r_{s_k, \varepsilon}(B_R)$ and 
\[
|I_{s_k, \varepsilon}(\varphi_{k}) - I_{\sigma,\varepsilon}(\varphi_\tau)| < \frac{\tau}{2}, \quad \forall k \geq \hat k(\tau).
\] 
As a consequence, we get that 
\[
|I_{s_k, \varepsilon}(\varphi_{k})-I_{\sigma,\varepsilon}(u_{\sigma, \varepsilon})|< \tau, \quad \forall k \geq \hat k(\tau).
\]
Therefore, since $u_{\sigma,\varepsilon}$ is a minimizer and $\varphi_{k} \in \mathcal{M}^r_{s_k, \varepsilon}(B_R)$, we infer that for all $k\geq \hat k(\tau)$
\[
C_{\mathcal{M}^r(B_R)}(s_k, \varepsilon) \leq I_{s_k, \varepsilon}(\varphi_{k}) \leq C_{\mathcal{M}^r(B_R)}(\sigma, \varepsilon) + \tau.
\]
Taking the $\limsup$ as $k\to+ \infty$ we get that
\[
\limsup_{k\to +\infty}C_{\mathcal{M}^r(B_R)}(s_k, \varepsilon) \leq C_{\mathcal{M}^r(B_R)}(\sigma, \varepsilon) + \tau,
\]
and since $\tau>0$ is arbitrary we obtain the desired result. The proof is then complete.
\end{proof}

In the next result we prove a uniform bound with respect to $s$ for the $L^\infty$-norm of the solutions. 
\begin{lemma}\label{L:linfbound}
Let $0<s_0< s_1\leq 1$, $n >2s_1$ and $\varepsilon \in (0, \hat \varepsilon)$, where $\hat \varepsilon$ is given by Lemma \ref{L:nodalenerasynt}. Then there exists $C>0$, depending on $\varepsilon$ but not on $s$, such that
\[
C^{-1} \leq \sup_{s \in [s_0, s_1)}|u_{s, \varepsilon}|_\infty\leq C,
\]
for every least energy radial sign-changing solution $u_{s, \varepsilon} \in \mathcal{M}^r_{s, \varepsilon}(B_R)$ of \eqref{fracsubcrit}.
\end{lemma}
\begin{proof}
Let us fix $s_0$, $s_1$, $n$ and $\ep$ as in the statement.
The first inequality is trivial. As for the second one, it can be proved in two different ways. Indeed, from \cite[Theorem 3.2]{IanMosSqu} there exists $M \in C(\R^+)$ such that 
\[
|u_{s, \varepsilon}|_\infty \leq M(|u_{s, \varepsilon}|_{2^*_s}).
\]
A careful analysis of the proof shows that the function $M$ can be chosen in such a way that $M$ depends only on $n, R, s_0$, $s_1$ and $\varepsilon$, but not on $s$. 
Since $u_{s, \varepsilon} \in \mathcal{M}_{s, \varepsilon}^r(B_R)\subset \mathcal{N}_{s, \varepsilon}(B_R)$ and $u_{s, \varepsilon}$ is a least energy sign-changing solution to \eqref{fracradsubcrit}, we infer that 
\begin{equation}\label{E:passag20}
C_{\mathcal{M}^r(B_R)}(s, \varepsilon) = \frac{2^*_s-2-\varepsilon}{2(2^*_s-\varepsilon)}\|u_{s, \varepsilon}\|^2_s.
\end{equation}
Thus, thanks to the fractional Sobolev embedding and Lemma \ref{L:nodalenerasynt} we deduce that $|u_{s, \varepsilon}|_{2^*_s} \leq C_1$, for some constant $C_1>0$ independent on $s$. Similarly, using that $2C_{\mathcal{N}(B_R)}(s, \varepsilon) \leq C_{\mathcal{M}^r(B_R)}(s, \varepsilon)$ and Lemma \ref{AsintPos} we obtain that $|u_{s, \varepsilon}|_{2^*_s} \geq C_0>0$, where $C_0$ does not depend on $s$, and the desired result easily follows.

Alternatively, we can argue as follows: fix $s_0$, $s_1$, $n$ and $\ep$ as in the statement. Since $u_{s, \ep}$ is a least energy sign-changing solution to \eqref{fracradsubcrit} with $u_{s, \ep}\in \mathcal{M}_{s, \varepsilon}^r(B_R)\subset \mathcal{N}_{s, \varepsilon}(B_R)$, and since Lemma \ref{L:nodalenerasynt} holds, by \eqref{E:passag20} we get that the quantity $|u_{s, \varepsilon}|^{2^*_{s}-\varepsilon}_{2^*_{s}-\varepsilon}$ is uniformly bounded with respect to $s \in [s_0, s_1)$. Now, suppose by contradiction that there exists a sequence $(s_k)_k \subset [s_0, s_1)$ and a sequence $(u_{s_k, \varepsilon})_k$ such that $\delta_{s_k} := |u_{s_k, \varepsilon}|_{\infty}\to +\infty$, as $k\to +\infty$. Up to a subsequence, $s_k \to \sigma \in [s_0, s_1]$, as $k\to +\infty$. Let us consider the rescaled functions
\[
v_k(x) := \frac{1}{\delta_{s_k}}u_{s_k, \varepsilon}\left(\frac{x}{\delta_{s_k}^{\beta_{s_k}}}\right), \quad x \in  \R^n
\] 
where $\beta_{s_k}:= \frac{2}{n-{2s_k}}$. We recall that 
\[
\|v_k\|_{s_k}^2 = \|u_{s_k, \varepsilon}\|_{s_k}^2, \quad \quad |v_k|^{2^*_{s_k}-\varepsilon}_{2^*_{s_k}-\varepsilon}=\left(\delta_{s_k}\right)^{-\ep \left(\frac{n-2s_k}{2s_k}\right)} | u_{s_k, \varepsilon}|^{2^*_{s_k}-\varepsilon}_{2^*_{s_k}-\varepsilon}, 
\]
and $v_k$ weakly satisfies 
\begin{equation}\label{E:rescalprob}
\begin{cases}
(-\Delta)^{s_k}v_k = \frac{1}{\delta_{s_k}^{\varepsilon}}|v_k|^{2^*_s-2-\varepsilon}v_k & \text{in }B_{\delta^{\beta_{s_k}}_s},\\
v_k = 0 & \text{in }B_{\delta_s^{\beta_{s_k}}}.
\end{cases}
\end{equation}
Arguing exactly as in Lemma \ref{L:maxpointbehavior} we see that $v_k \to v$ in $C^{0,\alpha}_{loc}(\R^n)$, for some $\alpha \in (0, s_0)$, where $v \not \equiv 0$. On the other hand, by Fatou's Lemma we have 
$$  |v|^{2^*_{\sigma}-\varepsilon}_{2^*_{\sigma}-\varepsilon} \leq \liminf_{k\to +\infty} |v_k|^{2^*_{s_k}-\varepsilon}_{2^*_{s_k}-\varepsilon}= \liminf_{k\to +\infty} \left(\delta_{s_k}\right)^{-\ep \left(\frac{n-2s_k}{2s_k}\right)} | u_{s_k, \varepsilon}|^{2^*_{s_k}-\varepsilon}_{2^*_{s_k}-\varepsilon}=0,$$
because $\delta_{s_k}\to +\infty$ and $|u_{s_k, \varepsilon}|^{2^*_{s_k}-\varepsilon}_{2^*_{s_k}-\varepsilon}$ is bounded. Hence $v\equiv 0$ and we get a contradiction. The proof is complete.
\end{proof}

In the next result we study the asymptotic behavior of the solutions as $s$ goes to some limit value.

\begin{lemma}\label{L:convstoone}
Let $0<s_0<s_1\leq 1$, $n >2s_1$ and $\varepsilon \in (0,\hat \varepsilon)$, where $\hat \varepsilon$ is given by Lemma \ref{L:nodalenerasynt}. Let $s_k \to \sigma$, where $\sigma \in [s_0,s_1]$, and let $(u_{s_k, \varepsilon})_k$ be a sequence of least energy nodal radial solution to \eqref{fracradsubcrit}. 
Then 
\[
u_{s_k, \varepsilon}\to u_{\sigma, \varepsilon} \text{ in }C^{0,s_0}_{loc}(\R^n), 
\]
where $u_{\sigma, \varepsilon} \in \mathcal{M}^r_{\sigma, \varepsilon}(B_R)$ 
weakly satisfies
\[
\begin{cases}
(-\Delta)^s u_{\sigma, \varepsilon} = |u_{\sigma, \varepsilon}|^{2^*_\sigma-2-\varepsilon}u_{\sigma, \varepsilon} & \text{ in }B_R, \\
u_{\sigma, \varepsilon} = 0 &\text{ in }\R^n \setminus B_R.
\end{cases}
\]
In addition, it holds that
\[
\lim_{s_k \to \sigma}I_{s_k, \varepsilon}(u_{s_k, \varepsilon}) = I_{\sigma, \varepsilon}(u_{\sigma, \varepsilon}).
\]
\end{lemma}
\begin{proof}
It suffices to argue as in \cite[Theorem 6.7]{CorIac}, taking into account Lemma \ref{L:nodalenerasynt} and Lemma \ref{L:linfbound}.
\end{proof}
As a corollary of the previous results we obtain the continuity of the map $s \mapsto C_{\mathcal{M}^r(B_R)}(s, \varepsilon)$.
\begin{proposition}\label{P:contenermap}
Let $0<s_0<s_1 \leq1$, $n >2s_1$ and let $\varepsilon \in (0,\check \varepsilon)$, where $\check \varepsilon$ is given by Lemma \ref{L:nodalchar}. Let $(s_k)_k \subset [s_0, s_1)$, $\sigma \in [s_0, s_1]$, $(u_{s_k, \varepsilon})_k$ and $u_{\sigma, \varepsilon}$ be as is Lemma \ref{L:convstoone}. Then $u_{\sigma, \varepsilon}$ is a least energy solution, that is, $I_{\sigma, \varepsilon}(u_{\sigma, \varepsilon}) = C_{\mathcal{M}^r(B_R)}(\sigma, \varepsilon)$.

In particular, for any $\varepsilon \in (0,\check \varepsilon)$ the map from $[s_0,s_1]$ to $\R$, defined by $s \mapsto C_{\mathcal{M}^r(B_R)}(s, \varepsilon)$, is continuous.  
\end{proposition}
\begin{proof}
Fixing $s_0$, $s_1$, $n$, $\ep$ as in the statement, applying both Lemma \ref{uppersemiconts} and Lemma \ref{L:convstoone}, since $0<\check \varepsilon < \hat \varepsilon$, we infer that
\[
C_{\mathcal{M}^r(B_R)}(\sigma, \varepsilon) \leq I_{\sigma, \varepsilon}(u_{\sigma, \varepsilon}) = \lim_{s \to \sigma}I_{s, \varepsilon}(u_{s, \varepsilon}) = \lim_{s \to \sigma }C_{\mathcal{M}^r(B_R)}(s, \varepsilon) \leq C_{\mathcal{M}^r(B_R)}(\sigma, \varepsilon),
\]
which implies both stated results. 
\end{proof}

The following Lemma grants that every least energy nodal radial solution in a ball changes sign exactly once, when $s$ is close to one. 
 
\begin{lemma}\label{L:oncenears}
Let $s_0 \in (0, 1)$ and $n\geq 3$. There exists $\varepsilon_0>0$ such that, for any $\varepsilon \in (0,\varepsilon_0)$, there exists $\bar s=\bar s(\varepsilon) \in (0,1)$ such that for any $s \in (\bar s, 1)$ any least energy radial sign-changing solution $u_{s, \varepsilon}$ to \eqref{fracradsubcrit} changes sign exactly once. 
\end{lemma}
\begin{proof}
We begin by recalling that, in the local case, when $n \geq 3$ there exists $\varepsilon_1 >0$ such that, for every $\varepsilon \in (0, \varepsilon_1)$, least energy radial sing-changing solutions to
\begin{equation}\label{E:locprob}
\begin{cases}
-\Delta u = |u|^{2^*_1-2-\varepsilon}u & \text{ in }B_R,\\
u = 0 & \text{ in }\R^n \setminus B_R,
\end{cases}
\end{equation}
change sign exactly once (see e.g. \cite{BEP3}). Now, let us fix $s_0 \in (0,1)$ and define $ \varepsilon_0 : = \min \{\check \varepsilon, \varepsilon_1\}$, where $\check \varepsilon$ is given by Lemma \ref{L:nodalchar} for $s_0$ and $s_1=1$.
Let us fix $\varepsilon \in (0, \varepsilon_0)$ and assume by contradiction that there exist $(s_k)_k \subset [s_0, 1)$ such that $s_k \to 1^-$ and a sequence $(u_{s_k, \varepsilon})_k$ of least energy radial sign-changing solutions in $B_R$ which change sign exactly twice for any $k$ (these functions change sign at most twice in view of Lemma \ref{L:nodalchar}). Then, by Proposition \ref{P:contenermap} we have that $u_{s_k, \varepsilon} \to u_{1, \varepsilon}$ in $C^{0,\alpha}_{loc}(\R^n)$, for some $\alpha \in (0,s_0)$, and that $u_{1, \varepsilon}$ is a least energy sing-changing solution to \eqref{E:locprob}. In particular, in view of our choice of $\ep$, $u_{1, \varepsilon}$ changes sign exactly once.

On the other hand, arguing as in the proof of \cite[Theorem 1.3]{CorIac}, we infer that the number of sign changes is preserved when passing to the limit as $s \to 1^-$ and thus $u_{1, \varepsilon}$ has to change sign twice. This gives a contradiction and concludes the proof. 
\end{proof}

Finally, we can state and prove Theorem \ref{mainteo2}. We first recall that, when speaking of a $L^2(B_R)$-continuous family $\mathcal{A} = \{v_{s, \varepsilon}\}_{s \in [s_0, 1)}$ of least energy nodal radial solutions to Problem \eqref{fracradsubcrit}, we mean a map $\Phi:[s_0,1) \to L^2(B_R)$ such that $\Phi$ is continuous and  $\Phi(s)=v_{s, \varepsilon} \in \mathcal{M}^r_{s,\ep}(B_R)$ is a least energy radial sign-changing solution  to Problem \eqref{fracradsubcrit} for any $s \in [s_0,1)$.

\begin{theorem}\label{mainteo2}
Let $s_0 \in (0, 1)$ and $n \geq 3$. There exists $\ep_0>0$ such that, for any $\ep \in (0, \ep_0)$, if there exists a $L^2(B_R)$-continuous family $\mathcal{A} = \{v_{s, \varepsilon}\}_{s \in [s_0, 1)}$ of least energy nodal radial solutions to Problem \eqref{fracradsubcrit}, then every element of the family changes sign exactly once. 
\end{theorem}
\begin{proof}
Let us fix $s_0\in (0,1)$, and let $ \varepsilon_0>0$ be the number given by Lemma \ref{L:oncenears}. Let us fix $\ep \in (0,\ep_0)$ and observe that, in view of Lemma \ref{L:oncenears}, there exists $\bar s \in (0,1)$ such that for any $s \in (\bar s, 1)$, every least energy radial sign-changing solution to \eqref{fracradsubcrit} changes sign only once. Let us fix $s_1 \in (\bar s, 1)$, let $\mathcal{A}$ be as in the statement, and set
\[
\mathfrak{S}_{\varepsilon} := \{s \in [s_0, s_1] \ ; \ v_{s, \varepsilon}  \text{ changes sign exactly once}\}.
\]

In view of the previous disccusion $\mathfrak{S}_{\varepsilon}$ is not empty. We claim that $\mathfrak{S}_{\varepsilon}$ is closed. 

Indeed, let $(s_k)_k \subset \mathfrak{S}_{\varepsilon}$ be a sequence such that $s_k \to \sigma$, for some $\sigma \in [s_0,s_1]$, and consider the associated sequence $(v_{s_k, \varepsilon})_k\subset \mathcal{A}$. By Lemma \ref{L:convstoone} and thanks to Proposition \ref{P:contenermap}, up to a subsequence, we have $v_{s_k, \varepsilon} \to u_\ep$ in $C^{0, \alpha}(\overline{B_R})$ for some $\alpha \in(0,s_0)$, where $u_\ep \in X^\sigma_0(B_R)$ is a least energy nodal radial solution of \eqref{fracradsubcrit} with $s=\sigma$. In particular, $v_{s_k, \varepsilon} \to u_\ep$ in $L^2(B_R)$ and, since we are assuming that $\mathcal{A}$ is $L^2(B_R)$-continuous, it holds that $u_\ep = v_{\sigma, \varepsilon} \in \mathcal{A}$. Now, taking into account Lemma \ref{L:nodalchar}, since $v_{s_k, \varepsilon} \to u_\ep$ in $C^{0, \alpha}(\overline{B_R})$ and  $v_{s_k, \varepsilon}$ changes sign once for all $k$, we infer that the only possibility is that $v_{\sigma, \varepsilon}$ changes sign only once. Hence $\sigma \in \mathfrak{S}_{\varepsilon}$, and the claim is proved.
\medskip

We claim that $\mathfrak{S}_{s_0, \varepsilon}$ is open. To prove the claim we show that the complementary set $\mathfrak{S}_{\varepsilon}^\mathsf{c}$ is closed. By definition and thanks to Lemma \ref{L:nodalchar} we have
\[
\mathfrak{S}_{\varepsilon}^\complement = \{s \in [s_0, s_1] \ ; \ v_{s, \varepsilon} \text{ changes sign exactly twice}\}.\] 
Let $(s_k)_k \subset \mathfrak{S}_{\varepsilon}^\complement$ be such that $s_k \to \sigma$ for some $\sigma \in[s_0,s_1]$, as $k \to +\infty$. Arguing as before, up to a subsequence, we get that $v_{s_k, \varepsilon} \to v_{\sigma, \varepsilon}$ in $C^{0, \alpha}(\overline{B_R})$, for $\alpha \in (0,s_0)$.

Let us denote by  $0<r_k'<r_k''<R$ the nodes of $v_{s_k, \varepsilon}(r) = v_{s_k, \varepsilon}(x)$, $|x| = r$. 
We observe that $r'_k \not\to 0$. Indeed, if $r'_k \to 0$, as $k \to +\infty$, as a consequence of the $C^{0,\alpha}$-convergence we infer that $v_{\sigma, \varepsilon}(0) = 0$. But this contradicts Lemma \ref{L:nodalchar}, and we are done.

Secondly, we claim that $r'_k - r''_k \not\to 0$. Indeed, assume by contradiction that $r'_k - r''_k \to 0$. Thanks to Lemma \ref{L:linfbound} we get that
\begin{equation}\label{eq:contrfin}
\begin{aligned}
|v_{s_k, \varepsilon}^-|^{2^*_{s_k}-\varepsilon}_{2^*_{s_k}-\varepsilon}&=\int_{B_{r''_k}\setminus B_{r'_k}}|v_{s_k,\varepsilon}^-|^{2^*_{s_k}-\varepsilon}\de x\\
 &\leq C|v_{s_k, \varepsilon}|^{2^*_{s_k}-\varepsilon}_\infty\int_{r'_k}^{r''_k}\rho^{n-1}\de \rho \leq C\left( (r''_k)^n-(r_k')^n\right) \to 0.
\end{aligned}
\end{equation}
On the other hand, since $v_{s_k, \varepsilon} \in \mathcal{M}^r_{s_k, \varepsilon}(B_R)$ and  thanks to Lemma \ref{AsintPos} we find a constant $C>0$ independent on $k$ such that
\begin{equation}\label{technicality}
C^\varepsilon S_{s_k} \leq S_{s_k, \varepsilon} \leq \frac{\|v^-_{s_k, \varepsilon}\|^2_{s_k}}{|v_{s_k, \varepsilon}^-|^{2}_{2^*_{s_k}- \varepsilon}} < |v_{s_k, \varepsilon}|_{2^*_{s_k}-\varepsilon}^{2^*_{s_k}-2-\varepsilon}.
\end{equation}
Hence, $ |v_{s_k, \varepsilon}|_{2^*_{s_k}-\varepsilon}^{2^*_{s_k}-2-\varepsilon}$ is bounded away from zero and this contradicts \eqref{eq:contrfin}.  

It remains to prove that $r''_k \not \to R$. To this end, we first point out that, thanks \cite[Theorem 1.2]{RosSer}, it holds
\[
\left\|\frac{v_{s_k, \varepsilon}}{\delta^{s_k}}\right\|_{C^{0,\alpha}(\overline{B_R})} \leq C|v_{s_k, \varepsilon}|^{2^*_{s_k}-1-\varepsilon}_\infty,
\] 
where $0<\alpha < \min\{s, 1-s\}$,  
\[
\delta^s(x) := d(x, \partial B_R))^s = (R- |x|)^s \quad \text{ and }\quad \frac{v_{s_k, \varepsilon}}{\delta^{s_k}}(R) := \lim_{\tau \to 0}\frac{v_{s_k, \varepsilon}(R-\tau)}{\delta^{s_k}(R-\tau)}.
\]
A careful analysis of the proof shows that the constant $C>0$ is uniform for $s \in [s_0,s_1]$ because $s_1$ is strictly less than one. Moreover we can fix $\alpha$ by choosing $0<\alpha < \min \{s_0, 1-s_1\}$. 

Assume now by contradiction that $R- r''_k \to 0$ as $k \to +\infty$. Since $v_{s_k, \varepsilon}(r''_k) = 0$, using the previous estimate and Lemma \ref{L:linfbound}, we have
\begin{equation}\label{Eq:zerobound}
\left|\frac{v_{s_k, \varepsilon}}{\delta^{s_k}}(R)\right| = \left|\frac{v_{s_k, \varepsilon}}{\delta^{s_k}}(R)- \frac{v_{s_k, \varepsilon}(r''_k)}{\delta^{s_k}(r''_k)}\right|\leq C|R- r''_k|^\alpha \to 0.
\end{equation}
On the other hand, applying the fractional Pohozaev identity (see \cite{RosSerrI}) to \eqref{fracradsubcrit} we get that
\[
\frac{2n- (n- 2s_k)(2^*_{s_k}- \varepsilon)}{2^*_{s_k}- \varepsilon}|v_{s_k, \varepsilon}|^{2^*_{s_k}- \varepsilon}_{2^*_{s_k}- \varepsilon} = \Gamma(1+s_k)^2 R |\partial B_R|\left|\frac{u_{s_k, \varepsilon}}{\delta^{s_k}}(R)\right|^2
\]
which, together with \eqref{Eq:zerobound}, implies that $|v_{s_k, \varepsilon}|^{2^*_{s_k}- \varepsilon}_{2^*_{s_k}- \varepsilon} \to 0$ as $s_k \to \sigma$, thus contradicting \eqref{technicality}.

From this discussion it follows that $r'_k$ and $r''_k$ definitely stay in the interior of the domain, away from the origin and their distance does not tend to zero. Thanks to the $C^{0,\alpha}$-convergence and by Lemma \ref{L:nodalchar} we infer that also $v_{\sigma, \varepsilon}$ changes sign exactly twice. Hence $\sigma \in \mathfrak{S}_{\varepsilon}^C$, and thus $\mathfrak{S}_{\varepsilon}^C$ is a closed set. At the end, $ \mathfrak{S}_{\varepsilon}$ is not empty, it is both open and closed, and thus $ \mathfrak{S}_{\varepsilon}=[s_0,s_1]$. Since by construction $s_1>\bar s$ we conclude that every element of $\mathcal{A}$ changes sign exactly once. The proof is complete.
\end{proof}

\end{section}

\end{document}